\documentclass{article}

\usepackage{graphicx}
\usepackage{amsmath,amsthm,amssymb,amsfonts}
\usepackage{verbatim}
\usepackage{stmaryrd} 
\usepackage{hyperref}

\usepackage[colorinlistoftodos]{todonotes}

\usepackage{tikz}
\usepackage{tkz-berge, tkz-graph}
\tikzset{vertex/.style={circle,draw,fill,inner sep=0pt,minimum size=1mm}}

\theoremstyle{plain}
\newtheorem{thm}{Theorem}
\newtheorem{lem}[thm]{Lemma}
\newtheorem{prop}[thm]{Proposition}
\newtheorem{cor}[thm]{Corollary}

\theoremstyle{definition}
\newtheorem{definition}[thm]{Definition}
\newtheorem{exl}[thm]{Example}
\newtheorem{remark}[thm]{Remark}

\numberwithin{thm}{section}

\newcommand{\adj}{\leftrightarrow}

\def\N{{\mathbb N}}
\def\Z{{\mathbb Z}}
\def\R{{\mathbb R}}

\begin{document}
\title{Alternate Product Adjacencies in Digital Topology}
\author{Laurence Boxer
         \thanks{
    Department of Computer and Information Sciences,
    Niagara University,
    Niagara University, NY 14109, USA;
    and Department of Computer Science and Engineering,
    State University of New York at Buffalo.
    E-mail: boxer@niagara.edu
    }
}

\date{ }
\maketitle

\begin{abstract}
We study properties of Cartesian products of digital images, using a variety of
adjacencies that have appeared in the literature.

Key words and phrases: digital topology, digital image, retraction, approximate fixed point property, continuous multivalued function, shy map

2010 MSC: 54C99, 05C99
\end{abstract}

\section{Introduction}
We study various adjacency relations for Cartesian products of
multiple digital images. We are
particularly interested in
``product properties'' - properties that are
preserved by taking Cartesian products - 
and ``factor properties'' for which
possession by a Cartesian product of digital images implies possession of the property by the 
factors. Many of the properties examined in this paper were
considered in~\cite{Boxer16a} for
adjacencies based on the
normal product adjacency. We consider other adjacencies in this
paper, including the tensor product adjacency, the Cartesian product adjacency, and the
composition or lexicographic adjacency.

\section{Preliminaries}
\label{prelims}
Much of the material that appears in this section is quoted or paraphrased
from~\cite{Boxer16a,BoxSta16
}, and other papers cited in this section.

We use $\N$, $\Z$, and $\R$ to represent
the sets of natural numbers, integers, and 
real numbers, respectively,  

A digital image is a graph. Usually, we consider the vertex set of a digital image to be
a subset of $\Z^n$ for some $n \in \N$. Further, we often, although not always, restrict
our study of digital images to finite graphs.
We will assume familiarity with the topological theory of digital images. See, e.g., \cite{Boxer94} for many of the standard definitions. All digital images $X$ are assumed to carry their own adjacency relations (which may differ from one image to another). When we wish to emphasize the particular adjacency relation we write the image as $(X,\kappa)$, where $\kappa$ represents
the adjacency relation.

\subsection{Common adjacencies}
To denote
that $x$ and $y$ are $\kappa$-adjacent points of some
digital image, we use the notation
$x \adj_{\kappa} y$, or $x \adj y$ when $\kappa$ can be understood.

The $c_u$-adjacencies are commonly used.
Let $x,y \in \Z^n$, $x \neq y$. Let $u$ be an integer,
$1 \leq u \leq n$. We say $x$ and $y$ are $c_u$-adjacent,
$x \adj_{c_u} y$, if
\begin{itemize}
\item there are at most $u$ indices $i$ for which 
      $|x_i - y_i| = 1$, and
\item for all indices $j$ such that $|x_j - y_j| \neq 1$ we
      have $x_j=y_j$.
\end{itemize}
A $c_u$-adjacency is often denoted by the number of points
adjacent to a given point in $\Z^n$ using this adjacency.
E.g.,
\begin{itemize}
\item In $\Z^1$, $c_1$-adjacency is 2-adjacency.
\item In $\Z^2$, $c_1$-adjacency is 4-adjacency and
      $c_2$-adjacency is 8-adjacency.
\item In $\Z^3$, $c_1$-adjacency is 6-adjacency,
      $c_2$-adjacency is 18-adjacency, and $c_3$-adjacency
      is 26-adjacency.
\end{itemize}

For Cartesian products of digital images, the normal product adjacency (see
Definitions~\ref{NP-def} and~\ref{NP_u-def}) has been used in papers
including~\cite{Han05,Boxer06,BoxKar12,Boxer16a} (errors in~\cite{Han05} are corrected
in~\cite{Boxer06}).
The tensor product adjacency (see Definition~\ref{tensor-def}), Cartesian product
adjacency (see Definition~\ref{product-adj-def}), and the lexicographic adjacency
(see Definition~\ref{lexico}) have not to our knowledge been studied in digital topology, so
their respective roles in digital topology remain to be determined.

Given digital images or graphs $(X,\kappa)$ and $(Y,\lambda)$, the
{\em normal product adjacency} $NP(\kappa,\lambda)$, also called
the {\em strong product adjacency}
(denoted $\kappa_*(\kappa,\lambda)$ in~\cite{BoxKar12}) generated by
$\kappa$ and $\lambda$ on the Cartesian product $X \times Y$ is defined
as follows.

\begin{definition}
\label{NP-def}
\rm{\cite{Berge,vL-W}}
Let $x, x' \in X$, $y, y' \in Y$.
Then $(x,y)$ and $(x',y')$ are $NP(\kappa,\lambda)$-adjacent in $X \times Y$
if and only if
\begin{itemize}
\item $x=x'$ and $y \adj_{\lambda} y'$; or
\item $x \adj_{\kappa} x'$ and $y=y'$; or
\item $x \adj_{\kappa} x'$ and $y \adj_{\lambda} y'$. \qed
\end{itemize}
\end{definition}

As a generalization of Definition~\ref{NP-def},
we have the following.

\begin{definition}
\label{NP_u-def}
\rm{\cite{Boxer16a}}
Let $u$ and $v$ be positive integers, $1 < u \leq v$. Let $\{(X_i,\kappa_i)\}_{i=1}^v$ be
digital images. Let $NP_u(\kappa_1, \ldots, \kappa_v)$ be the adjacency
defined on the Cartesian product $\Pi_{i=1}^v X_i$ as follows.
For $x_i,x_i' \in X_i$, $p=(x_1, \ldots, x_v)$ and $q=(x_1', \ldots, x_v')$ are
$NP_u(\kappa_1, \ldots, \kappa_v)$-adjacent if and only if
\begin{itemize}
\item for at least 1 and at most $u$ indices $i$, $x_i \adj_{\kappa_i} x_i'$, and
\item for all other indices $i$, $x_i=x_i'$. \qed
\end{itemize}
\end{definition}

\begin{definition}
\label{tensor-def}
{\rm \cite{Harary&Trauth}}
The {\em tensor product adjacency} on the
Cartesian product $\Pi_{i=1}^v X_i$ of $(X_i,\kappa_i)$,
denoted $T(\kappa_1,\ldots, \kappa_v)$, is as
follows. Given $x_i,x_i' \in X_i$, we have
$(x_1,\ldots, x_v)$ and $(x_1',\ldots,x_v')$ are
$T(\kappa_1,\ldots, \kappa_v)$-adjacent in
$\Pi_{i=1}^v X_i$ if and only if for all~$i$,
$x_i \adj_{\kappa_i} x_i'$. $\qed$
\end{definition}

\begin{figure}
\includegraphics[height=1.75in]{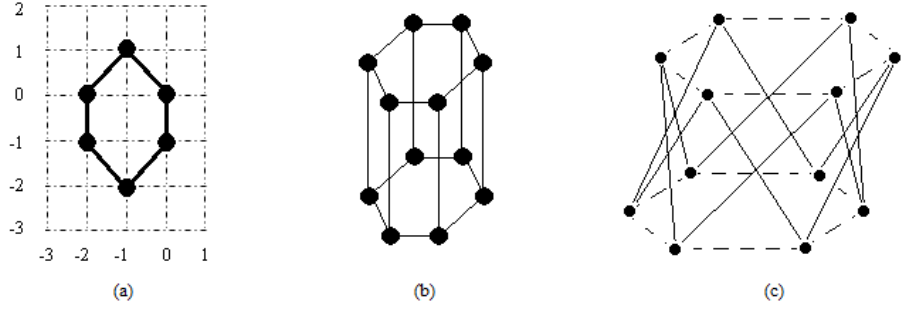}
\label{msc8-fig}
\caption{A digital simple closed
curve and its Cartesian product with
$[0,1]_{\Z}$.
(a) shows the simple
closed curve $MSC_8 \subset (\Z^2,c_2)$~\cite{Han03}.
(b) shows the set $MSC_8 \times [0,1]_{\Z} \subset \Z^3$ with either the $c_2 \times c_1$- or the $NP_1(c_2,c_1)$-adjacency.
(c) shows the set $MSC_8 \times [0,1]_{\Z} \subset \Z^3$ with the
$T(c_2, c_1)$-adjacency, where
adjacencies are shown by the solid lines. If the points of $MSC_8$ are circularly labeled $p_0, \ldots, p_5$,
then the $T(c_2,c_1)$-neighbors of
$(p_i, t)$ are $(p_{(i-1)\mod 6},1-t)$
and $(p_{(i+1)\mod 6},1-t)$, $t \in \{0,1\}$.
}
\end{figure}

\begin{definition}
{\rm \cite{Sabidussi60}}
\label{product-adj-def}
The {\em Cartesian product adjacency} on the
Cartesian product $\Pi_{i=1}^v X_i$ of $(X_i,\kappa_i)$,
denoted $\times_{i=1}^v \kappa_i$ or
$\kappa_1 \times \ldots \times \kappa_v$, is as
follows. Given $x_i,x_i' \in X_i$, we have
$(x_1,\ldots, x_v)$ and $(x_1',\ldots,x_v')$ are
$\times_{i=1}^v \kappa_i$-adjacent in
$\Pi_{i=1}^v X_i$ if and only if for some ~$i$,
$x_i \adj_{\kappa_i} x_i'$, and
for all indices $j \neq i$, $x_j = x_j'$. $\qed$
\end{definition}

The following has an elementary proof.

\begin{prop}
\label{CartesianAndNP}
For $\Pi_{i=1}^v (X_i,\kappa_i)$,
$\times_{i=1}^v \kappa_i = NP_1(\kappa_1,\ldots,\kappa_v)$. $\Box$
\end{prop}

\begin{definition}
\label{lexico}
\rm{\cite{Harary}}
Let $(X_i,\kappa_i)$ be digital images,
$1 \le i \le v$. Let $x_i, x_i' \in X_i$.
Let $p=(x_1,\ldots,x_v)$,
$p'=(x_1',\ldots,x_v')$.
We say $p$ and $p'$ are adjacent in
the {\em composition} or {\em lexicographic} adjacency on $\Pi_{i=1}^v X_i$ if $x_1 \adj_{\kappa_1} x_1'$, or if for some
index $j$, $1 \le j < v$, we have
$(x_1,\ldots, x_j) = (x_1',\ldots,x_j')$
and $x_{j+1} \adj_{\kappa_{j+1}} x_{j+1}'$. The
adjacency is denoted
$L(\kappa_1,\ldots,\kappa_v)$. $\qed$
\end{definition}

\begin{figure}
\includegraphics[height=1.5in]{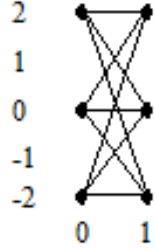}
\label{lex-adj-fig}
\caption{An illustration of lexicographic adjacency. This is
$[0,1]_{\Z} \times \{-2,0,2\}$, with
both factors regarded as subsets
of $(\Z,c_1)$, and
the $L(c_1,c_1)$ adjacency.}
\end{figure}

\begin{remark}
\label{gaps-allowed}
Notice that for $p$ and $p'$ to be
$L(\kappa_1,\ldots,\kappa_v)$-adjacent
with $x_k$ and $x_k'$ $\kappa_k$-adjacent, for indices
$m>k$ we do not require that
$x_m$ and $x_m'$ be either
equal or adjacent. See, e.g.,
Figure~
2,
where $(0,0)$ and $(1,2)$ are $L(c_1,c_1)$-adjacent.
This is unlike
other adjacencies discussed above.
$\qed$
\end{remark}

\subsection{Connectedness}
A subset $Y$ of a digital image $(X,\kappa)$ is
{\em $\kappa$-connected}~\cite{Rosenfeld},
or {\em connected} when $\kappa$
is understood, if for every pair of points $a,b \in Y$ there
exists a sequence $\{y_i\}_{i=0}^m \subset Y$ such that
$a=y_0$, $b=y_m$, and $y_i \adj_{\kappa} y_{i+1}$ for $0 \leq i < m$.

For two subsets $A,B\subset X$, we will say that $A$ and $B$ are \emph{adjacent} when there exist points $a\in A$ and $b\in B$ such that $a$ and $b$ are equal or adjacent. Thus sets with nonempty intersection are automatically adjacent, while disjoint sets may or may not be adjacent. It is easy to see that a finite union of connected adjacent sets is connected. 

\subsection{Continuous functions}
The following generalizes a definition of
~\cite{Rosenfeld}.

\begin{definition}
\label{continuous}
{\rm ~\cite{Boxer99}}
Let $(X,\kappa)$ and $(Y,\lambda)$ be digital images. A function
$f: X \rightarrow Y$ is $(\kappa,\lambda)$-continuous if for
every $\kappa$-connected $A$ of $X$ we have that
$f(A)$ is a $\lambda$-connected subset of $Y$. 
\end{definition}

When the adjacency relations are understood, we will simply say that $f$ is \emph{continuous}. Continuity can be reformulated in terms of adjacency of points:
\begin{thm}
\label{cont-by-adj}
{\rm ~\cite{Rosenfeld,Boxer99}}
A function $f:X\to Y$ is continuous if and only if, for any adjacent points $x,x'\in X$, the points $f(x)$ and $f(x')$ are equal or adjacent. \qed
\end{thm}

Note that similar notions appear
in~\cite{Chen94,Chen04} under the names
{\em immersion}, {\em gradually varied operator},
and {\em gradually varied mapping}.

\begin{thm}
\label{composition}
\rm{\cite{Boxer94,Boxer99}}
If $f: (A,\kappa) \to (B,\lambda)$ and $g: (B,\lambda) \to (C, \mu)$ are
continuous, then $g \circ f: (A,\kappa) \to (C, \mu)$ is continuous. \qed
\end{thm}

\begin{exl}
\rm{\cite{Rosenfeld}}
\label{const-exl}
A constant function between digital images is continuous. \qed
\end{exl}

\begin{exl}
\label{id-exl}
The identity function $1_X: (X, \kappa) \to (X, \kappa)$ is continuous. $\Box$
\end{exl}

\begin{definition}
Let $(X,\kappa)$ be a digital image
in $\Z^n$.
Let $x,y \in X$. A {\em 
$\kappa$-path of length $m$ from $x$ to $y$}
is a set $\{x_i\}_{i=0}^m \subset X$
such that $x=x_0$, $x_m=y$, and
$x_{i-1}$ and $x_i$ are equal or $\kappa$-adjacent for $1 \leq i \leq m$. If $x=y$, we say $\{x\}$ is a
{\em path of length 0 from $x$ to $x$}.
\end{definition}

Notice that for a path from $x$ to $y$ as
described above, the function
$f: [0,m]_{\Z} \to X$ defined by
$f(i)=x_i$ is $(c_1,\kappa)$-continuous. Such
a function is also called a {\em 
$\kappa$-path of length $m$ from $x$ to $y$}.

\subsection{Digital homotopy}
A homotopy between continuous functions may be thought of as
a continuous deformation of one of the functions into the 
other over a finite time period.

\begin{definition}{\rm (\cite{Boxer99}; see also \cite{Khalimsky})}
\label{htpy-2nd-def}
Let $(X,\kappa)$ and $(Y,\kappa')$ be digital images.
Let $f,g: X \rightarrow Y$ be $(\kappa,\kappa')$-continuous functions.
Suppose there is a positive integer $m$ and a function
$F: X \times [0,m]_{{\Z}} \rightarrow Y$
such that

\begin{itemize}
\item for all $x \in X$, $F(x,0) = f(x)$ and $F(x,m) = g(x)$;
\item for all $x \in X$, the induced function
      $F_x: [0,m]_{{\Z}} \rightarrow Y$ defined by
          \[ F_x(t) ~=~ F(x,t) \mbox{ for all } t \in [0,m]_{{\Z}} \]
          is $(2,\kappa')-$continuous. That is, $F_x(t)$ is a path in $Y$.
\item for all $t \in [0,m]_{{\Z}}$, the induced function
         $F_t: X \rightarrow Y$ defined by
          \[ F_t(x) ~=~ F(x,t) \mbox{ for all } x \in  X \]
          is $(\kappa,\kappa')-$continuous.
\end{itemize}
Then $F$ is a {\rm digital $(\kappa,\kappa')-$homotopy between} $f$ and
$g$, and $f$ and $g$ are {\rm digitally $(\kappa,\kappa')-$homotopic in} $Y$.
If for some $x_0 \in X$ we have $F(x_0,t)=F(x_0,0)$ for all
$t \in [0,m]_{{\Z}}$, we say $F$ {\rm holds $x_0$ fixed}, and $F$ is a
{\rm pointed homotopy}.
$\Box$
\end{definition}

We denote a pair of homotopic functions as
described above by $f \simeq_{\kappa,\kappa'} g$.
When the adjacency relations $\kappa$ and $\kappa'$ are understood in context,
we say $f$ and $g$ are {\em digitally homotopic} (or just {\em homotopic})
to abbreviate ``digitally 
$(\kappa,\kappa')-$homotopic in $Y$," and write
$f \simeq g$.

\begin{prop}
\label{htpy-equiv-rel}
{\rm ~\cite{Khalimsky,Boxer99}}
Digital homotopy is an equivalence relation among
digitally continuous functions $f: X \rightarrow Y$.
$\Box$
\end{prop}



\begin{definition}
{\rm ~\cite{Boxer05}}
\label{htpy-type}
Let $f: X \rightarrow Y$ be a $(\kappa,\kappa')$-continuous function and let
$g: Y \rightarrow X$ be a $(\kappa',\kappa)$-continuous function such that
\[ f \circ g \simeq_{\kappa',\kappa'} 1_X \mbox{ and }
   g \circ f \simeq_{\kappa,\kappa} 1_Y. \]
Then we say $X$ and $Y$ have the {\rm same $(\kappa,\kappa')$-homotopy type}
and that $X$ and $Y$ are $(\kappa,\kappa')$-{\rm homotopy equivalent}, denoted 
$X \simeq_{\kappa,\kappa'} Y$ or as
$X \simeq Y$ when $\kappa$ and $\kappa'$ are
understood.
If for some $x_0 \in X$ and $y_0 \in Y$ we have
$f(x_0)=y_0$, $g(y_0)=x_0$,
and there exists a homotopy between $f \circ g$
and $1_X$ that holds $x_0$ fixed, and 
a homotopy between $g \circ f$
and $1_Y$ that holds $y_0$ fixed, we say
$(X,x_0,\kappa)$ and $(Y,y_0,\kappa')$ are
{\rm pointed homotopy equivalent} and that $(X,x_0)$ 
and $(Y,y_0)$ have the 
{\rm same pointed homotopy type}, denoted 
$(X,x_0) \simeq_{\kappa,\kappa'} (Y,y_0)$ or as
$(X,x_0) \simeq (Y,y_0)$ when 
$\kappa$ and $\kappa'$ are understood.
$\Box$
\end{definition}

It is easily seen, from 
Proposition~\ref{htpy-equiv-rel}, that having the
same homotopy type (respectively, the same
pointed homotopy type) is an equivalence relation
among digital images (respectively, among pointed
digital images).

\subsection{Continuous and connectivity preserving multivalued functions}
Given sets $X$ and $Y$, a \emph{multivalued function} $f:X\to Y$ assigns a subset of $Y$ to each point of $x$. We will  write $f:X \multimap Y$. For $A \subset X$ and a multivalued function $f:X\multimap Y$, let $f(A) = \bigcup_{x \in A} f(x)$. 

\begin{definition}
\label{mildly}
\rm{\cite{Kovalevsky}}
A multivalued function $f:X\multimap Y$ is \emph{connectivity preserving} if $f(A)\subset Y$ is connected whenever $A\subset X$ is connected.
\end{definition}

As is the case with Definition \ref{continuous}, we can reformulate connectivity preservation in terms of adjacencies.

\begin{thm}
\rm{\cite{BoxSta16}}
\label{mildadj}
A multivalued function $f:X \multimap Y$ is connectivity preserving if and only if the following are satisfied:
\begin{itemize}
\item For every $x \in X$, $f(x)$ is a connected subset of $Y$.
\item For any adjacent points $x,x'\in X$, the sets $f(x)$ and $f(x')$ are adjacent. \qed
\end{itemize}
\end{thm}

Definition~\ref{mildly} is related to a definition of multivalued continuity for subsets of $\Z^n$ given and explored by Escribano, Giraldo, and Sastre in \cite{egs08, egs12} based on subdivisions. (These papers make a small error with respect to compositions, that is corrected in \cite{gs15}.) Their definitions are as follows:
\begin{definition}
For any positive integer $r$, the \emph{$r$-th subdivision} of $\Z^n$ is
\[ \Z_r^n = \{ (z_1/r, \dots, z_n/r) \mid z_i \in \Z \}. \]
An adjacency relation $\kappa$ on $\Z^n$ naturally induces an adjacency relation (which we also call $\kappa$) on $\Z_r^n$ as follows: $(z_1/r, \dots, z_n/r), (z_1'/r, \dots, z_n'/r)$ are adjacent in $\Z^n_r$ if and only if $(z_1, \dots, z_n)$ and $(z_1', \dots, z_n')$ are adjacent in $\Z^n$.

Given a digital image $(X,\kappa) \subset (\Z^n,\kappa)$, the \emph{$r$-th subdivision} of $X$ is 
\[ S(X,r) = \{ (x_1,\dots, x_n) \in \Z^n_r \mid (\lfloor x_1 \rfloor, \dots, \lfloor x_n \rfloor) \in X \}. \]

Let $E_r:S(X,r) \to X$ be the natural map sending $(x_1,\dots,x_n) \in S(X,r)$ to $(\lfloor x_1 \rfloor, \dots, \lfloor x_n \rfloor)$. \qed
\end{definition}

\begin{definition}
For a digital image $(X,\kappa) \subset (\Z^n,\kappa)$, a function $f:S(X,r) \to Y$ \emph{induces a multivalued function $F:X\multimap Y$} if $x \in X$ implies
\[ F(x) = \bigcup_{x' \in E^{-1}_r(x)} \{f(x')\}. \qed \]
\end{definition}

\begin{definition}
\label{multi-cont}
A multivalued function $F:X\multimap Y$ is called \emph{continuous} when there is some $r$ such that $F$ is induced by some single valued continuous function $f:S(X,r) \to Y$. 
\qed
\end{definition}

\begin{figure}
\includegraphics[height=1.25in]{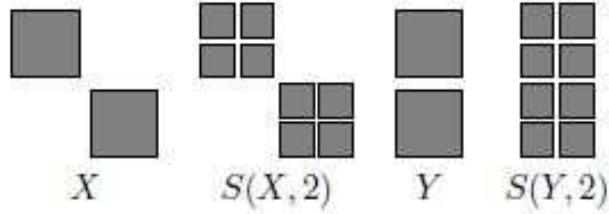}
\begin{center}
\caption{\cite{BoxSta16} Two images $X$ and $Y$ with their second subdivisions. (Subdivisions are drawn at half-scale.)
\label{subdivfig}}
\end{center}
\end{figure}

Note~\cite{BoxSta16} that the subdivision construction (and thus the notion of continuity) depends on the particular embedding of $X$ as a subset of $\Z^n$. In particular we may have $X, Y \subset \Z^n$ with $X$ isomorphic to $Y$ but $S(X,r)$ not isomorphic to $S(Y,r)$. 
E.g., in Figure~\ref{subdivfig}, when we use 8-adjacency for all images, $X$ and $Y$ are isomorphic, each being a set of two adjacent points, but $S(X,2)$ and $S(Y,2)$ are not isomorphic since $S(X,2)$ can be disconnected by removing a single point, while this is impossible in $S(Y,2)$. 

The definition of connectivity preservation makes no reference to $X$ as being embedded inside of any particular integer lattice $\Z^n$.

\begin{prop}
\label{pt-images-connected}
\rm{\cite{egs08,egs12}}
Let $F:X\multimap Y$ be a continuous multivalued function
between digital images. Then
\begin{itemize}
\item for all $x \in X$, $F(x)$ is connected; and
\item for all connected subsets $A$ of $X$, $F(A)$ is connected.
\qed
\end{itemize}
\end{prop}

\begin{thm}
\label{cont-hierarchy}
\rm{\cite{BoxSta16}}
For $(X,\kappa) \subset (\Z^n,\kappa)$, if $F:X\multimap Y$ 
is a continuous multivalued function, then $F$ is connectivity preserving. \qed
\end{thm}

The subdivision machinery often makes it difficult to prove that a given multivalued function is continuous. By contrast, many maps can easily be shown to be connectivity preserving. 

\subsection{Other notions of multivalued continuity}
Other notions of continuity have been given
for multivalued functions between graphs (equivalently,
between digital images). We have the following.

\begin{definition}
\rm{~\cite{Tsaur}}
\label{Tsaur-def}
Let $F: X \multimap Y$ be a multivalued function between
digital images.
\begin{itemize}
\item $F$ has {\em weak continuity} if for each pair of
      adjacent $x,y \in X$, $f(x)$ and $f(y)$ are adjacent
      subsets of $Y$.
\item $F$ has {\em strong continuity} if for each pair of
      adjacent $x,y \in X$, every point of $f(x)$ is adjacent
      or equal to some point of $f(y)$ and every point of 
      $f(y)$ is adjacent or equal to some point of $f(x)$.
     \qed
\end{itemize}
\end{definition}

\begin{prop}
\label{mild-and-weak}
\rm{\cite{BoxSta16}}
Let $F: X \multimap Y$ be a multivalued function between
digital images. Then $F$ is connectivity preserving if and
only if $F$ has weak continuity and for all $x \in X$,
$F(x)$ is connected. \qed
\end{prop}

\begin{exl}
\rm{\cite{BoxSta16}}
\label{pt-images-discon}
If $F: [0,1]_{\Z} \multimap [0,2]_{\Z}$ is defined by
$F(0)=\{0,2\}$, $F(1)=\{1\}$, then $F$ has both weak and
strong continuity. Thus a multivalued function between
digital images that has weak or strong continuity need not
have connected point-images. By Theorem~\ref{mildadj} and
Proposition~\ref{pt-images-connected} it
follows that neither having weak continuity nor having
strong continuity implies that a multivalued function is
connectivity preserving or continuous.
$\Box$
\end{exl}

\begin{exl}
\rm{\cite{BoxSta16}}
Let $F: [0,1]_{\Z} \multimap [0,2]_{\Z}$ be defined by
$F(0)=\{0,1\}$, $F(1)=\{2\}$. Then $F$ is continuous and
has weak continuity but
does not have strong continuity. $\Box$
\end{exl}

\begin{prop}
\rm{\cite{BoxSta16}}
Let $F: X \multimap Y$ be a multivalued function between
digital images. If $F$ has strong continuity and for
each $x \in X$, $F(x)$ is connected, then $F$ is
connectivity preserving. \qed
\end{prop}

The following shows that not requiring the image of a
point $F(p)$ to be connected can yield topologically unsatisfying 
consequences for weak and strong continuity.

\begin{exl}
\rm{\cite{BoxSta16}}
Let $X$ and $Y$ be nonempty digital images. Let
the multivalued function $f: X \multimap Y$ be defined by
$f(x)=Y$ for all $x \in X$.
\begin{itemize}
\item $f$ has both weak and strong continuity.
\item $f$ is connectivity preserving if and only if $Y$ is
      connected. \qed
\end{itemize}
\end{exl}

As a specific example~\cite{BoxSta16} consider $X= \{0\} \subset \Z$ and $Y = \{0,2\}$, all with $c_1$ adjacency. Then the function $F:X \multimap Y$ with $F(0) = Y$ has both weak and strong continuity, even though it maps a connected image surjectively onto a disconnected image.

\subsection{Shy maps and their inverses}
\begin{definition}
\label{shy-def}
\cite{Boxer05}
Let $f: X \rightarrow Y$ be a
continuous surjection of digital images. We say $f$ is
{\em shy} if
\begin{itemize}
\item for each $y \in Y$, $f^{-1}(y)$ is connected, and
\item for every $y_0,y_1 \in Y$ such that $y_0$ and $y_1$ are
      adjacent, $f^{-1}(\{y_0,y_1\})$ is
      connected. \qed
\end{itemize}
\end{definition}

Shy maps induce surjections on fundamental groups
~\cite{Boxer05}.
Some relationships between shy maps $f$ and their inverses
$f^{-1}$ as multivalued functions were studied in
~\cite{Boxer14,BoxSta16,Boxer16}.
Shyness as a factor or product property for the normal product adjacency was studied in~\cite{Boxer16a}.
We have the following.

\begin{thm}
\label{shy-thm}
\rm{\cite{BoxSta16,Boxer16}}
Let $f: X \to Y$ be a 
continuous surjection between digital images.
Then the following are equivalent.
\begin{itemize}
\item $f$ is a shy map.
\item For every connected $Y_0 \subset Y$ , $f^{-1}(Y_0)$
      is a connected subset of $X$.
\item $f^{-1}: Y \multimap X$ is a connectivity preserving multi-valued function.
\item $f^{-1}: Y \multimap X$ is a multi-valued function with weak continuity   
      such that for all $y \in Y$, $f^{-1}(y)$ is a connected subset of $X$. 
      \qed
\end{itemize}
\end{thm}

\subsection{Other tools}
Other terminology we use includes the following.
Given a digital image $(X,\kappa) \subset \Z^n$ and $x \in X$, the set of points adjacent to $x \in \Z^n$ and the
neighborhood of $x$ in $\Z^n$
are, respectively,
\[N_{\kappa}(x) = \{y \in \Z^n \, | \, y \mbox{ is }
    \kappa\mbox{-adjacent to }x\},\]
\[N_{\kappa}^*(x) = N_{\kappa}(x) \cup \{x\}.
\]

\section{Maps on products}
In this section, we consider various product
adjacencies with respect to continuity of functions.

\subsection{General properties}
\begin{definition}
Let $\kappa_1$ and $\kappa_2$ be adjacency
relations on a set $X$. We say
$\kappa_1$ {\em dominates} $\kappa_2$, $\kappa_1 \ge_d \kappa_2$,
or $\kappa_2$ {\em is dominated by} $\kappa_1$, $\kappa_2 \le_d \kappa_1$,
if for $x,x' \in X$, if $x$ and $x'$ are $\kappa_1$-adjacent then $x$ and $x'$ are $\kappa_2$-adjacent.
\end{definition}

\begin{exl}
We have the following comparisons of adjacencies.
\begin{itemize}
\item For $X \subset \Z^n$ and $1 \le u \le v \le n$,
$c_u \ge_d c_v$.
\item For $\Pi_{i=1}^v (X_i,\kappa_i)$ and $1 \le u \le v \le n$,
\[NP_u(\kappa_1,\ldots \kappa_v) \ge_d NP_v(\kappa_1,\ldots \kappa_v).
\]
\item For $\Pi_{i=1}^v (X_i,\kappa_i)$,
$T(\kappa_1,\ldots \kappa_v) \ge_d NP_v(\kappa_1,\ldots \kappa_v)$.
\item For $\Pi_{i=1}^v (X_i,\kappa_i)$, we have:
\begin{itemize}
\item $NP_u(\kappa_1,\ldots,\kappa_v) \ge_d L(\kappa_1,\ldots,\kappa_v)$ for $1 \le u \le v$;
\item $T(\kappa_1,\ldots,\kappa_v) \ge_d L(\kappa_1,\ldots,\kappa_v)$;
\item $\times_{i=1}^v \kappa_i \ge_d L(\kappa_1,\ldots,\kappa_v)$.
\end{itemize}
\end{itemize}
\end{exl}

\begin{proof} These follow
immediately from the
definitions of these adjacencies.
\end{proof}

The next example shows that
there are adjacencies that can
be applied to the same set $X$ such that neither dominates the other.

\begin{exl}
\label{neither-dominates}
In $X=\Z^6 = \Z^3 \times \Z^3$,
neither of $T(c_2,c_2)$ nor
$T(c_1,c_3)$ dominates the other.
\end{exl}

\begin{proof}
Consider the points
$p=(0,0,0,0,0,0)$ and
$q=(1,1,0,1,1,0)$. We have
$p \adj_{T(c_2,c_2)} q$ but
$p$ and $q$ are not $T(c_1,c_3)$-adjacent. Therefore
$T(c_2,c_2)$ does not dominate
$T(c_1,c_3)$.

Now consider $r=(1,0,0,1,1,1)$.
We have $p \adj_{T(c_1,c_3)} r$
but $p$ and $r$ are not
$T(c_2,c_2)$-adjacent. Therefore
$T(c_1,c_3)$ does not dominate
$T(c_2,c_2)$.
\end{proof}

Domination, and being dominated,
are transitive relations among
the adjacencies of a graph. I.e., we have the following.

\begin{prop} Given adjacencies
$\kappa$, $\lambda$, $\mu$ for a graph, if $\kappa \le_d \lambda$ 
and $\lambda \le_d \mu$, then
$\kappa \le_d \mu$.
\end{prop}

\begin{proof}
Elementary, and left to the reader.
\end{proof}

\begin{prop}
Let $f: X \to Y$ be a function.
\begin{itemize}
\item Let $\lambda_1$ and $\lambda_2$ be adjacency relations on $Y$. If $f$ is $(\kappa, \lambda_1)$ continuous and $\lambda_1 \ge_d \lambda_2$, then $f$ is
$(\kappa, \lambda_2)$ continuous.
\item Let $\kappa_1$ and $\kappa_2$ be adjacency relations on $X$. If $f$ is $(\kappa_1, \lambda)$-continuous and $\kappa_1 \le_d \kappa_2$, then $f$ is
$(\kappa_2, \lambda)$-continuous.
\end{itemize}
\end{prop}

\begin{proof} The assertions follows from the definitions of continuity and the $\ge_d$ relation.
\end{proof}

Given functions $f_i: (X_i, \kappa_i) \to (Y_i, \lambda_i)$,
$1 < i \leq v$, the function
\[ \Pi_{i=1}^v f_i: \Pi_{i=1}^v X_i \to \Pi_{i=1}^v Y_i\]
is defined by
\[ (\Pi_{i=1}^v f_i)(x_1, \ldots, x_v) = (f_1(x_1), \ldots, f_v(x_v)), \mbox{ where } x_i \in X_i. \]

\subsection{Normal product}
Here, we recall continuity properties of the normal product adjacency.

\begin{thm}
\label{prod-cont}
{\rm \cite{Boxer16a}}
Let $f_i: (X_i, \kappa_i) \to (Y_i, \lambda_i)$, $1 < i \leq v$.
Then the product map
\[ f=\Pi_{i=1}^v f_i: (\Pi_{i=1}^v X_i, NP_v(\kappa_1, \ldots, \kappa_v)) \to (\Pi_{i=1}^v Y_i, NP_v(\lambda_1, \ldots, \lambda_v)) \]
is continuous if and only if each $f_i$ is continuous. $\qed$
\end{thm}

\begin{thm}
\label{prod-iso}
\rm{\cite{Boxer16a}}
Let $X = \Pi_{i=1}^v X_i$.
Let $f_i: (X_i,\kappa_i) \to (Y_i\,\lambda_i)$, $1 \leq i \leq v$.
\begin{itemize}
\item For $1 \leq u \leq v$, if
the product map
$f=\Pi_{i=1}^v f_i: (X, NP_u(\kappa_1,\ldots,\kappa_v)) \to 
(\Pi_{i=1}^v Y_i, NP_u(\lambda_1,\ldots,\kappa_v))$ is an isomorphism, then
for $1 \leq i \leq v$, $f_i$ is an
isomorphism.
\item If $f_i$ is an isomorphism for all $i$, then the product map
$f=\Pi_{i=1}^v f_i: (X, NP_v(\kappa_1,\ldots,\kappa_v)) \to 
(\Pi_{i=1}^v Y_i, NP_v(\lambda_1,\ldots,\kappa_v))$ is an isomorphism. $\qed$
\end{itemize}
\end{thm}

\begin{thm}
\label{projection-cont}
{\rm \cite{Han05,Boxer16a}}
The projection maps
$p_i: (\Pi_{j=1}^v X_j, NP_u(\kappa_1, \ldots, \kappa_v)) \to (X_i, \kappa_i)$
defined by $p_i(x_1, \ldots, x_v) = x_i$
for $x_i \in (X_i,\kappa_i)$, are all continuous, for $1 \leq u \leq v$. $\qed$
\end{thm}

\subsection{Tensor product}
For the tensor product adjacency, we have the following.

\begin{prop}
\label{no-1zs}
Suppose $X=\Pi_{i=1}^v X_i$ has a pair of $T(\kappa_1,\ldots,\kappa_v)$-adjacent points. Then
\begin{itemize}
\item each $X_i$ has 2 $\kappa_i$-adjacent points;
      and
\item If $f: (X,T(\kappa_1,\ldots,\kappa_v)) \to 
      (\Pi_{j=1}^w Y_j, T(\lambda_1,\ldots,\lambda_w))$
      is continuous and not constant on some
      component of $X$, then for every $j$, $Y_j$
      has 2 $\lambda_j$-adjacent points.      
\end{itemize}
\end{prop}

\begin{proof} Let 
$p=(x_1,\ldots,x_v)$ and $p'=(x_1',\ldots,x_v')$
be $T(\kappa_1,\ldots,\kappa_v)$-adjacent in $X$.
Then for each $i$, $x_i$ and $x_i'$ are
$\kappa_i$-adjacent in $X_i$, which establishes
the first assertion. Further, if $f$ is
as hypothesized, the continuity of $f$ implies
there are $T(\kappa_1,\ldots,\kappa_v)$-adjacent
$p,p'$ such that
$f(p)=(y_1,\ldots, y_w)$ and $f(p')=(y_1',\ldots, y_w')$ are unequal, hence $T(\lambda_1,\ldots,\lambda_w)$-adjacent. Therefore, for all~$j$,
$y_j$ and $y_j'$ are $\lambda_j$-adjacent.
\end{proof}

It is easy to construct examples showing that the assertions obtained
from Proposition~\ref{no-1zs} by
substituting the normal product
adjacency $NP_v$ for $T$ are false.

\begin{thm}
\label{prod-cont-implies-factor}
Let $X=\Pi_{i=1}^v X_i$, $Y=\Pi_{i=1}^v Y_i$.
If the product map
\[ f=\Pi_{i=1}^v f_i: (X, T(\kappa_1,\ldots,\kappa_v)) \to (Y, T(\lambda_1, \ldots, \lambda_v)) \]
is continuous, then for each~$i$,
$f_i: (X_i,\kappa_i) \to (Y_i,\lambda_i)$ is
continuous.
\end{thm}

\begin{proof}
If $x_i, x_i'$ are $\kappa_i$-adjacent
in $X_i$, then
$p=(x_1,\ldots,x_v)$ and $p'=(x_1',\ldots,x_v')$
are $T(\kappa_1,\ldots,\kappa_v)$-adjacent
in $X$. Thus $f(p)$ and $f(p')$ are equal or
$T(\lambda_1,\ldots,\lambda_v)$-adjacent in $Y$.
This implies $f_i(x_i)$ and $f_i(x_i')$ are
equal or $\lambda_i$-adjacent in $Y_i$.
Thus $f_i$ is continuous.
\end{proof}

However, the converse to
Theorem~\ref{prod-cont-implies-factor} is
not generally true, as shown in the following.

\begin{exl}
\label{need-local-1-1}
Let $f: [0,1]_{\Z} \to [0,1]_{\Z}$ be the identity
function. Let $g: [0,1]_{\Z} \to [0,1]_{\Z}$ be the
constant function $g(x)=0$. Then, using
Examples~\ref{id-exl} and~\ref{const-exl}, $f$ and $g$ are
each $(c_1,c_1)$-continuous, but
$f \times g: [0,1]_{\Z} \times [0,1]_{\Z} \to
[0,1]_{\Z} \times [0,1]_{\Z}$ is not
$(T(c_1,c_1),T(c_1,c_1))$-continuous.
\end{exl}

\begin{proof} This follows from the observations
that $(0,0)$ and $(1,1)$ are $T(c_1,c_1)$-adjacent,
but $(f \times g)(0,0)=(0,0)$ and $(f \times g)(1,1)=(1,0)$ are 
neither equal nor $T(c_1,c_1)$-adjacent.
\end{proof}

A partial converse to
Theorem~\ref{prod-cont-implies-factor} is obtained by using
the following notion.

\begin{definition}
A continuous function
$f: (X, \kappa) \to (Y, \lambda)$ is
{\em locally one-to-one} if
$f|_{N_{\kappa}^*(x,1)}$ is one-to-one for all
$x \in X$. $\qed$
\end{definition}

Note any function between digital images that is one-to-one must be locally one-to-one.

\begin{thm}
\label{T-prod-continuity}
Suppose $f_i: (X_i, \kappa_i) \to (Y_i,\lambda_i)$ is continuous and locally
one-to-one for $1 \le i \le v$. Then the product function
$f=\Pi_{i=1}^v f_i: \Pi_{i=1}^v X_i \to \Pi_{i=1}^v Y_i$ is $(T(\kappa_1,\ldots, \kappa_v), T(\lambda_1,\ldots,\lambda_v))$-continuous and locally one-to-one.
\end{thm}

\begin{proof}
Suppose $f_i: (X_i, \kappa_i) \to (Y_i,\lambda_i)$ is continuous and 
locally one-to-one for $1 \le i \le v$. Let $p=(x_1,\ldots, x_v)$
and $p'=(x_1',\ldots, x_v')$ be
$T(\kappa_1,\ldots, \kappa_v)$-adjacent,
where $x_i$ and $x_i'$ are
$\kappa_i$-adjacent in $X_i$. Since
$f_i$ is continuous and locally one-to-one,
we must have that $f_i(x_i)$ and $f_i(x_i')$
are $\lambda_i$-adjacent in $Y_i$. Thus,
$f(p)$ and $f(p')$ are
$T(\lambda_1,\ldots,\lambda_v)$-adjacent, so
$f$ is continuous and locally one-to-one.
\end{proof}

\begin{thm}
\label{T-prod-cont}
Let $X=\Pi_{i=1}^v X_i$, $Y=\Pi_{i=1}^v Y_i$.
Then the product map
\[ f=\Pi_{i=1}^v f_i: (X, T(\kappa_1,\ldots,\kappa_v)) \to (Y, T(\lambda_1, \ldots, \lambda_v)) \]
is an isomorphism if and only if each
$f_i$ is an isomorphism.
\end{thm}

\begin{proof}
If $f$ is an isomorphism, each $f_i$ must
be one-to-one and onto. Therefore,
$f_i^{-1}: Y_i \to X_i$ is a single-valued
function.

By Theorem~\ref{prod-cont-implies-factor}, each $f_i$ is continuous. Since
$f^{-1}=\Pi_{i=1}^v f_i^{-1}$, it follows
from Theorem~\ref{prod-cont-implies-factor}
that each $f_i^{-1}$ is continuous. Hence
$f_i$ is an isomorphism.

Conversely, if each $f_i$ is an isomorphism,
then $f$ is one-to-one and onto, so
$f^{-1}=\Pi_{i=1}^v f_i^{-1}$ is a
single-valued function. By Theorem~\ref{T-prod-continuity}, $f$ is continuous.

Similarly, $f^{-1}$ is continuous.
Therefore, $f$ is an isomorphism.
\end{proof}

\begin{thm}
\label{T-projection-cont}
The projection maps
$p_i: (\Pi_{i=1}^v X_i, T(\kappa_1, \ldots, \kappa_v)) \to (X_i, \kappa_i)$
defined by $p_i(x_1, \ldots, x_v) = x_i$
for $x_i \in X_i$ are all continuous.
\end{thm}

\begin{proof} Let $p=(x_1, \ldots, x_v)$ and
$p'=(x_1', \ldots,x_v')$ be
$T(\kappa_1, \ldots, \kappa_v)$-adjacent in
$\Pi_{i=1}^v X_i$, 
where $x_i,x_i' \in X_i$. Then for all indices $i$,
$x_i=p_i(p)$ and $x_i'=p_i(p')$ are
$\kappa_i$-adjacent. Thus, $p_i$ is continuous.
\end{proof}

A seeming oddity is that a common method of
injection that is often continuous, is not continuous when the tensor product
adjacency is used, as shown in the following.

\begin{prop}
\label{natural-inject-not-cont}
Let $(X,\kappa)$ and $(Y,\lambda)$ be
digital images. Let $y \in Y$. If $X$ has a pair of $\kappa$-adjacent points, then the
function $f: X \to (X \times Y,T(\kappa,\lambda))$ defined by
$f(x)=(x,y)$ is not continuous.
\end{prop}

\begin{proof} This is because given $\kappa$-adjacent
$x,x' \in X$, $f(x)=(x,y)$ and
$f(x')=(x',y)$ are not $T(\kappa,\lambda)$-adjacent.
\end{proof}

\subsection{Cartesian product}
\begin{thm}
\label{Cart-prod-cont}
Let $f_i: (X_i,\kappa_i) \to (Y_i,\lambda_i)$ be functions
between digital images, $1 \le i \le v$.
Let $X=\Pi_{i=1}^v X_i$, $Y=\Pi_{i=1}^v Y_i$. 
Then the product
function $f=\Pi_{i=1}^v f_i: X \to Y$ is $(\times_{i=1}^v \kappa_i, \times_{i=1}^v \lambda_i)$-continuous if 
and only if each $f_i$ is continuous.
\end{thm}

\begin{proof}
Suppose $f$ is continuous.
Let $x_i \adj_{\kappa_i} x_i'$ in
$X_i$. Let
$p=(x_1,\ldots,x_v)$,
$p'=(x_1,\ldots,x_{i-1},x_i',x_{i+1},\ldots,x_v)$.
Then $p \adj_{\times_{i=1}^v \kappa_i} p'$, so either $f(p)=f(p')$ or $f(p) \adj_{\times_{i=1}^v \lambda_i} f(p')$. The former case implies
$f_i(x_i)=f_i(x_i')$ and the latter
case implies $f_i(x_i) \adj_{\lambda_i} f_i(x_i')$. Hence
$f_i$ is continuous.

Suppose each $f_i$ is continuous.
Let $p$ and $p'$ be
$\times_{i=1}^v \kappa_i$-adjacent points of
$X$. Then there is only one index $k$ in which $p$ and $p'$ differ, i.e., for some
$x_i \in X_i$ and $x_k' \in X_k$,
$p=(x_1, \ldots, x_v)$,
$p'=(x_1, \ldots,x_{k-1},x_k', x_{k+1}, \ldots, x_v)$, and
$x_k \adj_{\kappa_k}x_k'$.
Then $f(p)$ and $f(p')$ have the
same $i^{th}$ coordinate for
$i \neq k$, and have
$k^{th}$ coordinates of
$f_k(x_k)$ and $f_k(x_k')$, respectively. Continuity of
$f_k$ implies either $f_k(x_k)=f_k(x_k')$ or
$f_k(x_k) \adj_{\kappa_k} f_k(x_k')$. Therefore, $f$ is
continuous.
\end{proof}

\begin{thm}
\label{product-projection-cont}
The projection maps
$p_i: (\Pi_{i=1}^v X_i, \times_{i=1}^v \kappa_i) \to (X_i, \kappa_i)$
defined by $p_i(x_1, \ldots, x_v) = x_i$
for $x_i \in X_i$ are all continuous.
\end{thm}

\begin{proof}
This follows from Proposition~\ref{CartesianAndNP} and Theorem~\ref{projection-cont}.
\end{proof}

By contrast with Proposition~\ref{natural-inject-not-cont}, we have the following.

\begin{prop}
\label{natural-inject-product-cont}
Let $(X_i,\kappa_i)$ be digital images, $1 \le i \le v$.
Let $x_i \in X_i$. The
functions $I_i: X_i \to (\Pi_{i=1}^v X_i,\times_{i=1}^v \kappa_i)$ defined by
\[ I_i(x)= \left \{ \begin{array}{ll}
(x,x_2, \ldots, x_v) & \mbox{for } i=1; \\
(x_1, \ldots, x_{i-1}, x, x_{i+1} \ldots,
 x_v) & \mbox{for } 1 < i < v; \\
 (x_1, \ldots, x_{v-1}, x) & \mbox{for } i=v,
\end{array} \right .
\]
are continuous.
\end{prop}

\begin{proof}
This follows immediately from
Definition~\ref{product-adj-def}.
\end{proof}

\begin{thm}
\label{product-cont}
Let $X=\Pi_{i=1}^v X_i$, $Y=\Pi_{i=1}^v Y_i$.
Then the product map
\[ f=\Pi_{i=1}^v f_i: (X, \times_{i=1}^v \kappa_i) \to (Y, \times_{i=1}^v \lambda_i) \]
is an isomorphism if and only if 
each $f_i$ is an isomorphism,.
\end{thm}

\begin{proof}
Suppose $f$ is an isomorphism.
Then it follows from Proposition~\ref{CartesianAndNP} and 
Theorem~\ref{prod-iso} that $f_i$ is an isomorphism.

Suppose each $f_i$ is an isomorphism. Then
$f$ must be one-to-one and onto, and by Theorem~\ref{Cart-prod-cont}, $f$ is continuous.
Similarly, $f^{-1}=\Pi_{i=1}^v f_i^{-1}$ is continuous. Therefore,
$f$ is an isomorphism.
\end{proof}

\subsection{Lexicographic adjacency}
\begin{thm}
\label{lexico-1-1}
Suppose $f_i: (X_i,\kappa_i) \to (Y_i,\lambda_i)$ is a function
between digital images, $1 \le i \le v$. Let $f=\Pi_{i=1}^v f_i: \Pi_{i=1}^v X_i \to \Pi_{i=1}^v Y_i$ be the product function.
\begin{itemize}
\item If $f$ is $(L(\kappa_1,\ldots,\kappa_v),L(\lambda_1,\ldots,\lambda_v))$-continuous,
then each $f_i$ is $(\kappa_i,\lambda_i)$-continuous. Further,
if $f$ is locally one-to-one, then each $f_i$ is locally one-to-one.
\item If each $f_i$ is a continuous function that is
locally one-to-one, then
$f$ is
$(L(\kappa_1,\ldots,\kappa_v),L(\lambda_1,\ldots,\lambda_v))$-continuous.
\end{itemize}
\end{thm}

\begin{proof}
Suppose $f$ is $(L(\kappa_1,\ldots,\kappa_v),L(\lambda_1,\ldots,\lambda_v))$-continuous.
Let $x_i,x_i' \in X_i$ such that
$x_i \adj_{\kappa_i} x_i'$. Let
$p_0=(x_1,x_2, \ldots, x_v)$ and 
let
\[ p_i = \left \{ \begin{array}{ll}
(x_1', x_2, \ldots, x_v) & \mbox{for } i=1; \\
(x_1, \ldots, x_{i-1}, x_i',
 x_{i+1}, \ldots, x_v) & \mbox{for } 1 < i < v; \\
 (x_1, \ldots, x_{v-1}, x_v') &
 \mbox{for } i=v.
\end{array} \right .   
\]
Notice
\begin{equation}
\label{Lex-observ}
\mbox{$p_0$ and $p_i$ differ
only at index $i$ and
$p_0 \adj_{L(\kappa_1,\ldots,\kappa_v)} p_i$ for $1 \le i \le v$.}
\end{equation}
Therefore,
$f(p_0)$ and $f(p_i)$ are
$L(\lambda_1,\ldots,\lambda_v)$-adjacent or equal. It follows from
statement~(\ref{Lex-observ}) that $f_i(x_i)$ and $f_i(x_i')$
are $\lambda_i$-adjacent or equal. Since $\{x_i,x_i'\}$ is
an arbitrary set of $\kappa_i$-adjacent members of
$X_i$, $f_i$ is $(\kappa_i,\lambda_i)$-continuous. Further
if $f$ is locally one-to-one,
then from statement~(\ref{Lex-observ}), $f_i(x_i)$ and $f_i(x_i')$
are not equal, so $f_i$ is locally one-to-one.

Suppose each $f_i$ is continuous
and locally one-to-one.
Let $p,p' \in X = \Pi_{i=1}^v X_i$, where
$p=(x_1,\ldots,x_v)$, $p'=(x_1',\ldots,x_v')$, for
$x_i,x_i' \in X_i$. Assume
$p \adj_{L(\kappa_1,\ldots,\kappa_v)} p'$. Let $k$ be
the smallest index such that
$x_k \adj_{\kappa_k} x_k'$.
Since $f_k$ is locally
one-to-one,
\begin{equation}
\label{adj-values}
f_k(x_k) \adj_{\lambda_k} f_k(x_k').
\end{equation}
\begin{itemize}
\item If $k=1$, it follows
from Definition~\ref{lexico} that
$f(p) \adj_{L(\lambda_1,\ldots,\lambda_v)} f(p')$.
\item Otherwise, $i < k$ implies
      $x_i = x_i'$, hence
      $f_i(x_i) = f_i(x_i')$.
      Together with statement~(\ref{adj-values}), this implies
      $f(p) \adj_{L(\lambda_1,\ldots,\lambda_v)} f(p')$.
\end{itemize}
Then $f$ is
$(L(\kappa_1,\ldots,\kappa_v),L(\lambda_1,\ldots,\lambda_v))$-continuous, since $p$ and $p'$ were arbitrarily chosen.
\end{proof}

The following example illustrates
the importance of the locally
one-to-one hypothesis in
Theorem~\ref{lexico-1-1}.

\begin{exl}
\label{lexico-factor-not-implies-prod}
Let $X_i = [0,i]_{\Z}$ for
$i \in \{1,2\}$. Let
$f: X_1 \to X_2$ be the constant function with value $0$.
Then $f$ and $1_{X_2}$ are
$(c_1,c_1)$ continuous.
However, $f \times 1_{X_2}: X_1 \times X_2 \to X_2^2$ is not
$(L(c_1,c_1),L(c_1,c_1))$-continuous.
\end{exl}

\begin{proof}
Consider the points $p=(0,0)$ and $p'=(1,2)$. These points are
$L(c_1,c_1)$-adjacent in $X_1 \times X_2$. However,
$(f \times 1_{X_2})(p) = (0,0)$
and $(f \times 1_{X_2})(p') = (0,2)$ are neither equal nor
$L(c_1,c_1)$-adjacent in $X_2^2$.
\end{proof}

\begin{thm}
\label{L-iso}
Suppose $f_i: (X_i,\kappa_i) \to (Y_i,\lambda_i)$ is a function
between digital images, $1 \le i \le v$. Let $f=\Pi_{i=1}^v f_i: \Pi_{i=1}^v X_i \to \Pi_{i=1}^v Y_i$ be the product function.
Then $f$ is an
$(L(\kappa_1,\ldots,\kappa_v),L(\lambda_1,\ldots,\lambda_v))$-isomorphism if and
only if each $f_i$ is a
$(\kappa_i,\lambda_i)$-isomorphism.
\end{thm}

\begin{proof}
This follows easily from
Theorem~\ref{lexico-1-1}.
\end{proof}

\begin{prop}
The projection map
$p_1: (\Pi_{i=1}^v X_i, L(\kappa_1,\ldots,\kappa_v)) \to (X_1,\kappa_1)$ is continuous.
\end{prop}

\begin{proof}
Let $p \adj_{L(\kappa_1,\ldots,\kappa_v)} p'$ in $\Pi_{i=1}^v X_i$. Then 
$p=(x_1,\ldots,x_v)$,
$p'=(x_1',\ldots,x_v')$ for some
$x_i,x_i' \in X_i$, where
either $x_1=x_1'$ or $x_1 \adj_{\kappa_1} x_1'$. Since
$p_1(p)=x_1$ and $p_1(p')=x_1'$,
it follows that $p_1$ is continuous.
\end{proof}

By contrast, we have the following.

\begin{exl}
\label{L-projection-not-cont}
Let $v>1$. The projection maps
$p_i: ([0,2]_{\Z}^v,L(c_1,\ldots,c_1)) \to ([0,2]_{\Z},c_1)$ are not
continuous for $1<i\le v$.
\end{exl}

\begin{proof} Let
$x=(0,0,\ldots,0)$, $y=(1,2,\ldots,2)$. Then
$x \adj_{L(c_1,\ldots,c_1)} y$ in
$[0,2]_{\Z}^v$, but $i>1$ implies
$p_i(x)=0$ and $p_i(y)=2$, which 
are not $c_1$-adjacent in $[0,2]_{\Z}$. The assertion follows.
\end{proof}

\subsection{More on isomorphisms}
We have the following.

\begin{thm}
\label{permute}
Let
$\sigma: \{i\}_{i=1}^v \to \{i\}_{i=1}^v$ be a permutation.
Let $f_i: (X_i,\kappa_i) \to (Y_{\sigma(i)},\lambda_{\sigma(i)})$ be an isomorphism
of digital images, $1 \le i \le v$. Let $1 \le u \le v$. 
Let $(\kappa,\lambda)$ be any of
\[(NP_u(\kappa_1,\ldots,\kappa_v),NP_u(\lambda_{\sigma(1)},\ldots,\lambda_{\sigma(v)})), \]
\[(T(\kappa_1,\ldots,\kappa_v),T(\lambda_{\sigma(1)},\ldots,\lambda_{\sigma(v)})), \mbox{ or} \]
\[(\times_{i=1}^v \kappa_i,\times_{i=1}^v \lambda_{\sigma(i)}). \]
Let $X=\Pi_{i=1}^v X_i$, $Y= \Pi_{i=1}^v Y_{\sigma(i)}$.
Then the function
$f: X \to Y$ defined by
\[ f(x_1,\ldots,x_v)=(f_1(x_1), \ldots,(f_v(x_v))
\]
is an isomorphism.
\end{thm}

\begin{proof}
It is easy to see that $f$ is one-to-one and onto. Continuity of $f$ and of 
$f^{-1}$ follows easily from the definitions of the adjacencies under
discussion. Thus, $f$ is an isomorphism.
\end{proof}

The following example shows that the
lexicographic adjacency does not yield
a conclusion analogous to that of 
Theorem~\ref{permute}.

\begin{exl}
\label{pretzel}
Let $X_1 = \{0,1\} \subset (\Z,c_1)$. Let $X_2 = \{0,2\}\subset (\Z,c_1)$. Then $X=(X_1 \times X_2, L(c_1,c_1))$ and $Y=(X_2 \times X_1, L(c_1,c_1))$ are not isomorphic.
\end{exl}

\begin{proof}
Observe that $X$ is connected, since the 4 points of $X$ form a path in the sequence
\[ (0,0),~(1,0),~(0,2),~(1,2)
\]
(see Figure~2).
However, $Y$ is not connected, as there
is no path in $Y$ from $(0,0)$ to $(2,0)$.
The assertion follows.
\end{proof}

\section{Connectedness}
In this section, we compare product 
adjacencies with respect to the
property of connectedness.

\begin{thm}
\label{prod-connected}
\rm{\cite{Boxer16a}}
Let $(X_i,\kappa_i)$ be digital images, $i \in \{1,2, \ldots, v\}$. Then
$(X_i,\kappa_i)$ is connected for all $i$ if and only
$(\Pi_{i=1}^v X_i, NP_v(\kappa_1, \ldots, \kappa_v))$ is connected. $\qed$
\end{thm}

\begin{thm}
\label{Pi-conn}
Let $(X_i,\kappa_i)$ be digital images, $i \in \{1,2, \ldots, v\}$.
If 
$\Pi_{i=1}^v X_i$ is $T(\kappa_1, \ldots, \kappa_v)$-connected, then $X_i$ is $\kappa_i$-connected for all $i$.
\end{thm}

\begin{proof}
These assertions follow from
Definition~\ref{continuous} and Theorem~\ref{T-projection-cont}.
\end{proof}

However, the converse to
Theorem~\ref{Pi-conn} is not
generally true, as shown by the
following.

\begin{exl}
\label{conn-not-preserved}
Let $X = \{0\} \subset \Z$,
$Y = [0,1]_{\Z} \subset \Z$. Then
$X$ and $Y$ are each $c_1$-connected.
However:
\begin{itemize}
\item $X \times Y = \{(0,0), (0,1)\}$ is not
$T(c_1,c_1)$-connected.
\item $Y \times Y$ has two $T(c_1,c_1)$-components,
$\{(0,0), (1,1)\}$ and $\{(1,0),(0,1)\}$. $\qed$
\end{itemize}
\end{exl}

See also 
Figure~1(c),
which illustrates that $MSC_8 \times [0,1]_{\Z}$ is
not $T(c_2,c_1)$-connected, although $MSC_8$ is $c_2$-connected and
$[0,1]_{\Z}$ is $c_1$-connected.

For the Cartesian product adjacency, we have the following.

\begin{thm}
\label{product-Pi-conn}
Let $(X_i,\kappa_i)$ be digital images, $i \in \{1,2, \ldots, v\}$. Then
$\Pi_{i=1}^v X_i$ is $\times_{i=1}^v \kappa_i$-connected if and only if $X_i$ is $\kappa_i$-connected for all $i$.
\end{thm}

\begin{proof}
Suppose $X=\Pi_{i=1}^v X_i$ is $\times_{i=1}^v \kappa_i$-connected. It follows from Proposition~\ref{product-projection-cont} that each $X_i$ is
$\kappa_i$-connected.

Suppose each $X_i$ is $\kappa_i$-connected. Let $p = (x_1, \ldots, x_v)$
and $p' = (x_1', \ldots, x_v')$ be points of $X$ such that
$x_i, x_i' \in X_i$. There are $\kappa_i$-paths $P_i$ in $X_i$ from
$x_i$ to $x_i'$. If the functions $I_i$ are as in
Proposition~\ref{natural-inject-product-cont}, then it is easily seen that
$\bigcup_{i=1}^v I_i(P_i)$ is a $\times_{i=1}^v \kappa_i$-path in $X$ from
$p$ to $p'$. Since $p$ and $p'$ were arbitrarily
chosen, it follows that $X$ is $\times_{i=1}^v \kappa_i$-connected.
\end{proof}

\begin{prop}
\label{Lex-1st-factor}
Let $(X,\kappa)$ and $(Y,\lambda)$
be digital images, such that $|X| > 1$. Then
$(X \times Y, L(\kappa,\lambda))$
is connected if and only if
$(X,\kappa)$ is connected.
\end{prop}

\begin{proof}
Suppose $(X,\kappa)$ is connected.
Let $p=(x,y)$, $p'=(x',y')$, with
$x,x' \in X$, $y,y' \in Y$.
\begin{itemize}
\item If $x=x'$ then, since $|X| > 1$ and $X$ is connected, there exists $x_0 \in X$ such that $x \adj_{\kappa} x_0$. Therefore,
$p \adj_{L(\kappa,\lambda)} (x_0,y) \adj_{L(\kappa,\lambda)}
(x,y')=p'$.
\item Suppose $x \neq x'$. Since
$X$ is connected, there
is a path in $X$, $P=\{x_i\}_{i=0}^n$, such that
\[x=x_0 \adj_{\kappa} x_1 \adj_{\kappa} \dots \adj_{\kappa} x_{n-1} \adj_{\kappa} x_n=x'.
\]
Therefore,
\[ p = (x_0,y) \adj_{L(\kappa,\lambda)} (x_1, y') \adj_{L(\kappa,\lambda)} (x_2,y')
\adj_{L(\kappa,\lambda)} \ldots \]
\[ \adj_{L(\kappa,\lambda)} (x_n,y')=p'.
\]
\end{itemize}
Therefore,
$(X \times Y, L(\kappa,\lambda))$
is connected.

Suppose $(X,\kappa)$ is not
connected. Then there exist
$x,x' \in X$ such that $x$ and
$x'$ are in distinct components
of $X$. Let $y,y' \in Y$. By
Definition~\ref{lexico},
there is no path in
$(X \times Y, L(\kappa,\lambda))$
from $(x,y)$ to $(x',y')$.
Therefore, $(X \times Y, L(\kappa,\lambda))$ is not connected.
\end{proof}

An argument similar to that used
for the proof of Proposition~\ref{Lex-1st-factor}
yields the following.

\begin{thm}
\label{lex-conn}
Let $(X_i,\kappa_i)$ be digital
images, $1 \leq i \le v$.
Suppose $k$ is the smallest
index for which $|X_k|>1$.
Then $(\Pi_{i=1}^v X_i, L(\kappa_1,\ldots,\kappa_v))$ is
connected if and only if
$(X_k,\kappa_k)$ is connected. $\qed$
\end{thm}

\section{Homotopy}
\subsection{Tensor product}
In~\cite{Boxer16a}, it is shown that many homotopy properties are
preserved by Cartesian products with the $NP_v$ adjacency. 
We show that we cannot make analogous claims for
the tensor product adjacency.

\begin{exl}
\label{htpy-not-preserved}
There are digital
images $(X_i,\kappa_i)$ and $(Y_i,\lambda_i)$ and continuous
functions $f_i,g_i: X_i \to Y_i$,
$i \in \{1,2\}$, such that
\[ f_i \simeq g_i \mbox{ but }
f_1 \times f_2 \not \simeq_{T(\kappa_1,\kappa_2),T(\lambda_1,\lambda_2)} g_1 \times g_2. \]
\end{exl}

\begin{proof}
We can use Example~\ref{conn-not-preserved}. E.g., if
$X_1=X_2=Y_1=Y_2=[0,1]_{\Z}$,
$f_1=f_2: X_1 \to Y_1$ is the identity
function, and $g_1=g_2: X_2 \to Y_2$ is the constant function taking the 
value 0, we have $f_1 \simeq_{c_1,c_1}g_1$ and 
$f_2 \simeq_{c_1,c_1}g_2$. As
we saw in Example~\ref{conn-not-preserved}, $[0,1]_{\Z}^2$
is not $T(c_1,c_1)$-connected, so its identity function $f_1 \times f_2$ is not 
homotopic to the constant function $g_1 \times g_2$.
\end{proof}

\begin{exl}
\label{htpy-type-not-preserved}
There are digital
images $(X_i,\kappa_i)$ and $(Y_i,\lambda_i)$ for $i \in \{1,2\}$, such that
$X_i$ and $Y_i$ have the same homotopy type,
but $(X_1 \times X_2,T(\kappa_1,\lambda_1))$ and 
$(Y_1 \times Y_2,T(\kappa_2,\lambda_2))$ do not have
the same homotopy type.
\end{exl}

\begin{proof}
We saw in Example~\ref{conn-not-preserved} that
$[0,1]_{\Z}^2$ is not $T(c_1,c_1)$-connected; however,
it is trivial that $\{0\}^2 = \{(0,0)\}$ is
$T(c_1,c_1)$-connected. Therefore, we can
take $X_1 = X_2 = [0,1]_{\Z} \subset (\Z,c_1)$,
$Y_1=Y_2=\{0\}\subset (\Z,c_1)$.
\end{proof}

\subsection{Cartesian product adjacency}
\begin{thm}
\label{Cart-prod-htpy}
Let $f_i, g_i: (X_i,\kappa_i) \to (Y_i,\lambda_i)$ be continuous functions between digital images, $1 \le i \le v$. Let $X = \Pi_{i=1}^v X_i$,
$Y  = \Pi_{i=1}^v Y_i$,
$f=\Pi_{i=1}^v f_i: X \to Y$, $g=\Pi_{i=1}^v g_i: X \to Y$.
Then $f \simeq_{\times_{i=1}^v \kappa_i, \times_{i=1}^v \lambda_i} g$ if
and only if for
all~$i$, $f_i \simeq_{\kappa_i,\lambda_i} g_i$. Further, $f$ and $g$ are pointed homotopic
if and only if for each~$i$, $f_i$ and $g_i$ are pointed homotopic.
\end{thm}

\begin{proof}
Suppose $f \simeq_{\times_{i=1}^v \kappa_i, \times_{i=1}^v \lambda_i} g$. Then 
there is a homotopy
\[ H: \Pi_{i=1}^v X_i \times [0,m]_{\Z} \to \Pi_{i=1}^v X_i \]
such that $H(p,0)=f(p)$ and
$H(p,m)=g(p)$ for all $p \in X$.
Let $x_i \in X_i$ and let $H_i: X_i \times [0,m]_{\Z} \to Y_i$ be defined by
\[ H_i(x,t) = p_i(H(I_i(x),t)),
\]
where $I_i$ is the continuous injection of Proposition~\ref{natural-inject-product-cont} corresponding to the point $(x_1,\ldots,x_v)\in X$ and $p_i$ is the continuous
projection map of Theorem~\ref{product-projection-cont}. Then
\[ H_i(x,0)=p_i(f(I_i(x))=f_i(x) \mbox{ and } H_i(x,m)=p_i(g(I_i(x))=g_i(x).
\]
Since the composition of continuous functions is continuous (Theorem~\ref{composition}), it follows that
$H_i$ is a homotopy from $f_i$ to $g_i$.
Further, if $H$ holds some point $p_0$ of $X$ fixed, then we can take $p_0=(x_1,\ldots,x_v)$ to be
the point of $X$ used in Proposition~\ref{natural-inject-product-cont},
and we can conclude that $H_i$ holds $p_i(p)=x_i$ fixed.

Suppose for all~$i$, $f_i \simeq_{\kappa_i,\lambda_i} g_i$. Let
$H_i: X_i \times [0,m_i]_{\Z} \to Y_i$ be a $(\kappa_i,\lambda_i)$-homotopy
from $f_i$ to $g_i$. We execute these homotopies ``one coordinate at a time," as
follows. For $x=(x_1,\ldots, x_v) \in X$ such that
$x_i \in X_i$, let $M_i = \sum_{k=1}^i m_i$ for all $i$ and let
$H: X \times [0, M_v]_{\Z} \to Y$ be defined by $H(x_1,\ldots,x_v,t) =$
\begin{itemize}
\item $(H_1(x_1,t),f_2(x_2), \ldots, f_v(x_v))$ if $0 \le t \le m_1$;
\item $(g_1(x_1) \ldots, g_{j-1}(x_{j-1}), H_j(x_j,t-M_{j-1}), f_{j+1}(x_{j+1}),
      \ldots, f_v(x_v))$ if $M_{j-1} \le t \le M_j$;
\item $(g_1(x_1) \ldots, g_{v-1}(x_{v-1}), H_v(x_v,t-M_{v-1}))$ if
      $M_{j-1} \le t \le M_j$.
\end{itemize}
It is easily seen that $H$ is well defined and is a homotopy from $f$ to $g$.

Further, if $H_i$ holds $x_i$ fixed, then $H$ holds $x$ fixed.
\end{proof}

\begin{cor}
\label{product-adj-ptd-htpy-equiv}
Let $(X_i,\kappa_i)$ and $(Y_i,\lambda_i)$
be digital images, $1 \le i \le v$.
Then $X=\Pi_{i=1}^v X_i$ and $Y=\Pi_{i=1}^v Y_i$ are $(\times_{i=1}^v \kappa_i, \times_{i=1}^v \lambda_i)$-(pointed) homotopy equivalent if and only if
for each $i$, $(X_i,\kappa_i)$ and $(Y_i,\lambda_i)$ are
(pointed) homotopy equivalent.
\end{cor}

\begin{proof}
This follows from Theorem~\ref{Cart-prod-htpy}
\end{proof}

\subsection{Lexicographic adjacency}
\begin{thm}
\label{lex-factor-and-prod}
Let $(X_i,\kappa_i)$ be digital
images for $1 \le i \le v$.
Let $X = \Pi_{i=1}^v X_i$.
If there is a smallest index~$k$
such that $|X_k|>1$, then
$(X,L(\kappa_1,\ldots,\kappa_v))$ and $(X_k, \kappa_k)$
have the same pointed homotopy type.
\end{thm}

\begin{proof}
For each $i \neq k$, let $x_i \in X_i$. Let
$I_k: X_k \to X$ be the injection of Proposition~\ref{natural-inject-product-cont}.
By choice of $k$, $I_k$ is
$(\kappa_k,L(\kappa_1,\ldots,\kappa_v))$-continuous. Also by
choice of $k$, the projection
map
$p_k:(X,L(\kappa_1,\ldots,\kappa_v)) \to (X_k, \kappa_k)$
is continuous. Notice
$p_k \circ I_k = 1_{X_k}$. Also,
the function
$H: X \times [0,1]_{\Z} \to X$
defined for $p=(y_1,\ldots,y_v) \in X$ with $y_i \in X_i$ by
\[ H(p,t) = \left \{ \begin{array}{ll}
p & \mbox{if } t=0; \\
(y_1,x_2, \ldots, x_v) & \mbox{if } t=1 \mbox{ and } k =1; \\
(x_1,\ldots,x_{k-1},y_k,x_{k+1},\ldots,x_v) & \mbox{if } t=1 \mbox{ and } 1 < k < v; \\
(x_1, \ldots, x_{v-1}, y_v) & \mbox{if } t=1 \mbox{ and } k=v,
\end{array} \right .
\]
is easily seen from the choice of~$k$ to be a homotopy
from $1_X$ to $I_k \circ p_k$ that holds fixed the point
$(x_1,\ldots,x_v)$. The
assertion follows.
\end{proof}

\begin{cor}
Let $(X,\kappa)$ and $(Y,\lambda)$ be digital images
of different homotopy types.
If $|X|>1$ and $|Y|>1$, then
$(X \times Y, L(\kappa,\lambda))$
and $(Y \times X, L(\lambda,\kappa))$ have different homotopy types.
\end{cor}

\begin{proof}
This follows immediately from
Theorem~\ref{lex-factor-and-prod}.
\end{proof}

\begin{cor}
Let $(X_i,\kappa_i)$ and
$(Y_i, \lambda_i)$ be digital
images, $1 \leq i \le v$. Let
$X = \Pi_{i=1}^v X_i$,
$Y = \Pi_{i=1}^v Y_i$.
Suppose there exist a smallest index $j$ such that $|X_j|>1$, and
a smallest index
$k$ such that $|Y_k|>1$. If
$(X_j, \kappa_j)$ and
$(Y_k, \kappa_k)$ have the same
(pointed) homotopy type, then
$(X, L(\kappa_1,\ldots,\kappa_v))$ and
$(Y, L(\lambda_1,\ldots,\lambda_v))$ have the same (pointed) homotopy type.
\end{cor}

\begin{proof}
By Theorem~\ref{lex-factor-and-prod},
$(X, L(\kappa_1,\ldots,\kappa_v))$ and 
$(X_j,\kappa_j)$ have the same pointed homotopy type, and
$(Y_k,\lambda_k)$ and $(Y,L(\lambda_1,\ldots,\lambda_v))$ have the same pointed homotopy type. Since we also have
assumed $(X_j,\kappa_j)$ and $(Y_k,\lambda_k)$ have the same (pointed) homotopy type, the
assertion follows from the
transitivity of (pointed) homotopy type.
\end{proof}

\section{Retractions}
\begin{definition}
\rm{\cite{Borsuk,Boxer94}}
\label{retract-def}
Let $Y \subset (X,\kappa)$. A $(\kappa,\kappa)$-continuous
function $r: X \to Y$ is a {\em retraction}, and 
$A$ is a {\em retract of} $X$, if $r(y)=y$ for all
$y \in Y$. \qed
\end{definition}

\begin{thm}
\label{ret-thm}
\rm{\cite{BoxSta16}}
Let $A_i \subset (X_i,\kappa_i)$, $i \in \{1, \ldots, v\}$. Then
$A_i$ is a retract of $X_i$ for all $i$ if and only if
$\Pi_{i=1}^v A_i$ is a retract of 
$(\Pi_{i=1}^v X_i, NP_v(\kappa_1, \ldots, \kappa_v))$.
$\qed$
\end{thm}

\subsection{Tensor product adjacency}
The following example shows that one of the
assertions obtained by using the tensor product
adjacency rather than $NP_v$ in Theorem~\ref{ret-thm} is not generally valid.

\begin{exl}
Let $X=\{(0,0),~(1,0),~(1,1)\}
\subset (\Z^2,c_2)$. Observe that
$X'=\{(0,0),~(1,0)\}$ is a
$c_2$-retract of $X$, and
$\{0\}$ is a $c_1$-retract of
$[0,1]_{\Z}$. However,
$X' \times \{0\}$ is not a
$T(c_2,c_1)$-retract of
$X \times [0,1]_{\Z}$.
\end{exl}

\begin{proof} Note
$X \times [0,1]_{\Z}$ is
$T(c_2,c_1)$-connected, since
\[(0,0,0),~(1,0,1),~(1,1,0),~(0,0,1),~(1,0,0),~(1,1,1)
\]
is a listing of its points in a
$T(c_2,c_1)$-path; but
$X' \times \{0\}= \{(0,0,0), (1,0,0)\}$ is not
$T(c_2,c_1)$-connected. The assertion follows.
\end{proof}

The question of whether
$\Pi_{i=1}^v A_i$ being a retract of 
$(\Pi_{i=1}^v X_i, T(\kappa_1, \ldots, \kappa_v))$ implies $A_i$ is a
$\kappa_i$-retract of $X_i$, for all~$i$,
is unknown at the current writing.

\subsection{Cartesian product adjacency}
For the Cartesian product adjacency, we have the following analog of Theorem~\ref{ret-thm}.

\begin{thm}
Suppose $A_i \subset (X_i,\kappa_i)$. 
Let $X=\Pi_{i=1}^v X_i$,
$A=\Pi_{i=1}^v A_i$.
Then there is a retraction $r_i: X_i \to A_i$, $1 \le i \le v$
if and only if there is a retraction
$r: (X, \times_{i=1}^v \kappa_i) \to
(A, \times_{i=1}^v \kappa_i)$.
\end{thm}

\begin{proof}
Suppose there is a retraction $r_i: X_i \to A_i$, $1 \le i \le v$. Let $r = \Pi_{i=1}^v r_i: X \to A$.
Clearly $r(x) \in A$ for all $x \in X$, and $r(a)=a$ for all $a \in A$.
By Theorem~\ref{Cart-prod-cont},
$r$ is continuous.
Therefore, $r$ is a retraction.

Conversely, suppose there exists a retraction
$r: (X, \times_{i=1}^v \kappa_i) \to (A, \times_{i=1}^v \kappa_i)$.
Let $r_i = p_i \circ r \circ I_i:
(X_i,\kappa_i) \to (A_i,\kappa_i)$, where $I_i$ is
the injection of Proposition~\ref{natural-inject-product-cont} and the $x_i$ of 
Proposition~\ref{natural-inject-product-cont} satisfies $x_i \in A_i$.
Since composition
preserves continuity, Theorem~\ref{product-projection-cont} and
Proposition~\ref{natural-inject-product-cont}
imply $r_i$ is continuous. Further, for $a_i \in A_i$ we clearly have
$r_i(a_i)=a_i$. Thus, $r_i$ is a retraction.
\end{proof}

\subsection{Lexicographic adjacency}
For the lexicographic adjacency, we do not have an
analog of Theorem~\ref{ret-thm},
as shown by the following example.

\begin{exl}
$\{0\}$ is a $c_1$-retract of
$[0,1]_{\Z}$ and $[1,4]_{\Z}$ is
a $c_1$-retract of $[0,5]_{\Z}$.
However, $A=\{0\} \times [1,4]_{\Z}$ is
not an $L(c_1,c_1)$-retract of
$X=[0,1]_{\Z} \times [0,5]_{\Z}$.
\end{exl}

\begin{proof} We give a proof by contradiction.
Suppose there is an $L(c_1,c_1)$-retraction $r:
[0,1]_{\Z} \times [0,5]_{\Z} \to
\{0\} \times [1,4]_{\Z}$. Notice
$p=(0,1) \adj_{L(c_1,c_1)} (1,5)=p'$.
Since $r(p)=p$, the continuity of $r$
requires that $r(p')=p$ or $r(p') \adj_{L(c_1,c_1)} p$, hence
\[ r(p') \in \{p,(0,2)\}.
\]
But also
$p' \adj_{L(c_1,c_1)} (0,4)=q$, and since $r(q)=q$, the continuity of $r$
similarly requires that
\[ r(p') \adj_{L(c_1,c_1)} \{q,(0,3)\}.
\]
Therefore,
\[ r(p') \in \{p,(0,2)\} \cap \{q,(0,3)\} = \emptyset.\]
Since this is impossible, no such
retraction $r$ can exist.
\end{proof}

\section{Approximate fixed point property}
Some material in this section is quoted or
paraphrased from~\cite{Boxer16a,BEKLL}.

In both topology and digital topology,
\begin{itemize}
\item a {\em fixed point} of a continuous function
      $f: X \to X$ is a point $x \in X$ satisfying $f(x)=x$;
\item if every continuous $f: X \to X$ has a fixed point,
      then $X$ has the {\em fixed point property} (FPP).
\end{itemize}
However, a digital image $X$ has the FPP if and only
if $X$ has a single point~\cite{BEKLL}. Therefore, it
turns out that the {\em approximate fixed point property} is
more interesting for digital images.

\begin{definition}
\rm{\cite{BEKLL}}
\label{approxFP}
A digital image $(X,\kappa)$ has the
{\em approximate fixed point property (AFPP)} if every
continuous $f: X \to X$ has an {\em approximate fixed point},
i.e., a point $x \in X$ such that $f(x)$ is equal or
$\kappa$-adjacent to $x$. \qed
\end{definition}

The following is a minor generalization 
of Theorem~5.10 of~\cite{BEKLL}.

\begin{thm}
\rm{\cite{Boxer16a}}
\label{NP-AFPP}
Let $(X_i,\kappa_i)$ be digital images, $1 \leq i \leq v$.
Then for any $u \in \Z$ such that
$1 \leq u \leq v$, if
$(\Pi_{i=1}^v X_i, NP_u(\kappa_1, \ldots, \kappa_v))$
has the AFPP then
$(X_i,\kappa_i)$ has the AFPP for all $i$.
\end{thm}

Determining whether
analogs of Theorem~\ref{NP-AFPP} for the
tensor product adjacency, or for the
Cartesian product adjacency, are generally
true, appear to be difficult problems.
The following examples show that the
analogs of converses to Theorem~\ref{NP-AFPP} for the
tensor product adjacency and for the
Cartesian product adjacency are not
generally true.

\begin{exl} Although $([0,1]_{\Z},c_1)$ has the AFPP~\cite{Rosenfeld}, $([0,1]_{\Z}^2,T(c_1,c_1))$ does not have
the AFPP.
\end{exl}

\begin{proof}
Consider the function $f: [0,1]_{\Z}^2 \to [0,1]_{\Z}^2$ defined by
$f(a,b)=(1-a,b)$, i.e.,
\[f(0,0)=(1,0),~~f(0,1)=(1,1),~~f(1,0)=(0,0),~~f(1,1)=(0,1).\]
One can easily check that $f$ is continuous and has no approximate fixed
point when the $T(c_1,c_1)$ adjacency is used.
\end{proof}

\begin{exl} Although $([0,1]_{\Z},c_1)$ has the AFPP, $([0,1]_{\Z}^2,c_1 \times c_1)$ does not have the AFPP.
\end{exl}

\begin{proof}
Consider the function $f: [0,1]_{\Z}^2 \to [0,1]_{\Z}^2$ defined by
$f(a,b)=(1-a,1-b)$, i.e.,
\[f(0,0)=(1,1),~~f(0,1)=(1,0),~~f(1,0)=(0,1),~~f(1,1)=(0,0).\]
One can easily check that $f$ is continuous and has no approximate fixed
point when the $c_1 \times c_1$ adjacency is used.
\end{proof}

We have the following.

\begin{thm}
\label{L-AFPP}
Let $(X_i,\kappa_i)$ be digital images,
$1 \le i \le v$. Suppose there is a smallest
index~$k$ such that $X_k$ is $\kappa_k$-connected and $|X_k|>1$. If 
the product $(\Pi_{i=1}^v X_i, L(\kappa_1,\ldots,\kappa_v))$ has the AFPP property, then $(X_k,\kappa_k)$ has the AFPP property.
\end{thm}

\begin{proof}
Let $X=\Pi_{i=1}^v X_i$.

Suppose the product $(X, L(\kappa_1,\ldots,\kappa_v))$ has the AFPP property.
Let $g: X_k \to X_k$ be $\kappa$-continuous. Let $x_i \in X_i$.
Notice this means $X_i=\{x_i\}$ for $i < k$.
Let $X = \Pi_{i=1}^v X_i$.
Let $G: X \to X$ be defined by
\[ G(y_1,\ldots,y_v)= \left \{ \begin{array}{ll}
(g(y_1),x_2,\ldots,x_v) & \mbox{if } k=1; \\
(x_1, \ldots, x_{k-1},g(y_k),x_{k+1},\ldots,x_v) & \mbox{if } 1 < k < v; \\
(x_1, \ldots, x_{v-1}, g(y_v)) &\mbox{if } k=v.
\end{array} \right .
\]
Since $g$ is $\kappa_k$-continuous,
our choice of $k$ implies
$G$ is $L(\kappa_1,\ldots,\kappa_v)$-continuous. By hypothesis,
there is a $p=(y_1',\ldots,y_v') \in X$
with $y_i' \in X_i$ such that
$G(p)=p$ or $G(p) \adj p$. Therefore,
either
\[ g(y_k)=p_k(G(p))=p_k(p)=y_k \mbox{ or }
g(y_k) \adj_{\kappa_k} y_k. \]
Thus, $y_k$ is an approximate fixed point for $g$.
\end{proof}

\section{Multivalued functions}
We study various product adjacencies with respect to properties of multivalued functions.

The following has an elementary proof.

\begin{prop}
\label{single-val-as-multi}
Let $f: (X,\kappa) \to (Y, \lambda)$ be a single-valued function between
digital images. Then the following are equivalent.
\begin{itemize}
\item $f$ is continuous.
\item As a multivalued function, $f$ has weak continuity.
\item As a multivalued function, $f$ has strong continuity. $\Box$
\end{itemize}
\end{prop}

For multivalued functions
$F_i: X_i \multimap Y_i$, $1 \leq i \leq v$,
define the product multivalued function
\[ \Pi_{i=1}^v F_i: \Pi_{i=1}^v X_i \multimap \Pi_{i=1}^v Y_i \]
by
\[ (\Pi_{i=1}^v F_i)(x_1,\ldots, x_v) =
   \Pi_{i=1}^v F_i(x_i).
\]

\subsection{Weak continuity}
For $NP_v$, we have the following
results.
\begin{thm}
\label{weak-prod}
\rm{\cite{Boxer16a}}
Let $F_i: (X_i, \kappa_i) \multimap (Y_i,\lambda_i)$ be multivalued functions for $1 \le i \le v$.
Let $X = \Pi_{i=1}^v X_i$, $Y = \Pi_{i=1}^v Y_i$, and $F = \Pi_{i=1}^v F_i: (X, NP_v(\kappa_1,\ldots,\kappa_v)) \multimap (Y, NP_v(\lambda_1,\ldots,\lambda_v))$. Then
$F$ has weak continuity if and only if each 
$F_i$ has weak continuity. $\qed$
\end{thm}

For the tensor product, we have the following.

\begin{thm}
\label{weak-prod-cont-implies-factor}
For each index $i$ such that $1 \le i \le v$,
let $f_i: (X_i,\kappa_i) \multimap (Y_i,\lambda_i)$ be a
multivalued map between digital images.
Let $X=\Pi_{i=1}^v X_i$, $Y=\Pi_{i=1}^v Y_i$.
If the product multivalued map
\[ f=\Pi_{i=1}^v f_i: (X, T(\kappa_1,\ldots,\kappa_v)) \multimap (Y, T(\lambda_1, \ldots, \lambda_v)) \]
has weak continuity, then for each~$i$,
$f_i$ has weak continuity.
\end{thm}

\begin{proof}
For all indices~$i$, let $x_i \adj_{\kappa_i} x_i'$ in
$X_i$. Then, in $X$, we have
$p=(x_1,\ldots,x_v) \adj_{T(\kappa_1,\ldots,\kappa_v)} p'=(x_1',\ldots,x_v')$. The weak
continuity of~$f$ implies
$f(p)$ and $f(p')$ are adjacent
subsets of $(Y,T(\lambda_1,\ldots,\lambda_v))$. Therefore,
there exist $y \in f(p)$ and
$y' \in f(p')$ such that
$y=y'$ or $y \adj_{T(\lambda_1,\ldots,\lambda_v)} y'$.

Now, $y=(y_1,\ldots,y_v)$ where
$y_i \in f_i(x_i)$, and $y'=(y_1',\ldots,y_v')$ where
$y_i' \in f_i(x_i')$. If $y=y'$
then we have $y_i=y_i'$ for all
indices~$i$. If $y \adj_{T(\lambda_1,\ldots,\lambda_v)} y'$ then we have
$y_i \adj_{\lambda_i} y_i'$ for all indices~$i$. In either case,
we have for all~$i$ that $f_i(x_i)$ and $f_i(x_i')$ are adjacent subsets
of $Y_i$. It follows that each $f_i$ has weak continuity.
\end{proof}

The converse of Theorem~\ref{weak-prod-cont-implies-factor} is not generally true, as shown by the following.

\begin{exl}
\label{tensor-non-weak-prod}
Let $f$ and $g$ be the single-valued functions of
Example~\ref{need-local-1-1}. By Proposition~\ref{single-val-as-multi},
$f$ and $g$ have weak continuity. However, Example~\ref{need-local-1-1} shows 
that $f \times g$ is not $(T(c_1,c_1),T(c_1,c_1))$-continuous, so by
Proposition~\ref{single-val-as-multi}, $f \times g$ does not have
$(T(c_1,c_1),T(c_1,c_1))$-weak continuity. $\Box$
\end{exl}

For the Cartesian product adjacency, we have the following.

\begin{thm}
\label{weak-Cart-prod-implies-factor}
Let $f_i: (X_i,\kappa_i) \multimap (Y_i,\lambda_i)$ be
multivalued maps between digital
images, $1 \le i \le v$.
Let $X=\Pi_{i=1}^v X_i$, $Y=\Pi_{i=1}^v Y_i$. Then the product multivalued map
\[ f=\Pi_{i=1}^v f_i: (X, \times_{i=1}^v \kappa_i) \multimap (Y, \times_{i=1}^v \lambda_i) \]
has weak continuity if and only if for each~$i$,
$f_i$ has weak continuity.
\end{thm}

\begin{proof}
Suppose $f$ has weak continuity. Let $x_i \adj_{\kappa_i} x_i'$ in $X_i$.
Let
\[ x=(x_1, \ldots, x_v) \in X,\]
\[x'=(x_1, \ldots, x_{j-1}, x_j', x_{j+1}, \ldots, x_v) \in X \mbox{ for some index } j.\]
We have $x \adj_{\times_{i=1}^v \kappa_i} x'$.
Therefore, there exist
\[y=(y_1,\ldots,y_v) \in f(x) = \Pi_{i=1}^v f_i(x_i),\] 
\[y'=(y_1', \ldots, y_v') \in f(x')=  \Pi_{i=1}^{j-1} f_i(x_i) \times f_j(x_j) \times \Pi_{i=j+1}^v f_i(x_i)\]
such that
$y \adj_{\times_{i=1}^v \lambda_i} y'$. Therefore, we have $y_j \in f_j(x_j)$,
$y_j' \in f_j(x_j')$, and $y_j=y_j'$ or $y_j \adj_{\lambda_j} y_j'$. Thus,
$f_j$ has weak continuity.

Suppose each $f_i$ has weak
continuity. Let
$p \adj_{\times_{i=1}^v \kappa_i} p'$ in $X$, where
$p = (x_1,\ldots,x_v)$,
$p' = (x_1',\ldots,x_v')$,
$x_i, x_i' \in X_i$, and, from
the definition of the $\times_{i=1}^v \kappa_i$ adjacency, there is one index
$j$ such that $x_j \adj_{\kappa_j} x_j'$ and for all indices $i \neq j$, $x_i=x_i'$ and therefore $f_i(x_i)=f_i(x_i')$. Since $f_j$
has weak continuity, there
exist $y_j \in f_j(x_j)$ and
$y_j' \in f_j(x_j')$ such that
$y_j=y_j'$ or $y_j \adj_{\lambda_j} y_j'$. For $i \neq j$ we can take $y_i \in f_i(x_i)$. Then
$y = (y_1,\ldots,y_v)$ and
$y'=(y_1,\ldots,y_{j-1},y_j',y_{j+1},\ldots,y_v)$ are
equal or $\times_{i=1}^v \lambda_i$-adjacent, and we have
$y \in f(p)$, $y' \in f(p')$.
Therefore, $f$ has weak continuity.
\end{proof}

For the lexicographic adjacency,
Example~\ref{lex-weakStrong-exl} below shows there is
no general product property for weak continuity, and
Example~\ref{lex-weakStrong-factor-exl} below shows there is not a general factor property
for weak continuity.

\subsection{Strong continuity}
\begin{thm}
\label{strong-prod}
\rm{\cite{Boxer16a}}
Let $F_i: (X_i, \kappa_i) \multimap (Y_i,\lambda_i)$ be multivalued functions for $1 \le i \le v$.
Let $X = \Pi_{i=1}^v X_i$, $Y = \Pi_{i=1}^v Y_i$, and $F = \Pi_{i=1}^v F_i: (X, NP_v(\kappa_1,\ldots,\kappa_v)) \multimap (Y, NP_v(\lambda_1,\ldots,\lambda_v))$. Then
$F$ has strong continuity if and only if each 
$F_i$ has strong continuity. $\qed$
\end{thm}

For the tensor product adjacency, we have the following.

\begin{thm}
\label{strong-tensor-cont-implies-factor}
Let $f_i: (X_i,\kappa_i) \multimap (Y_i,\lambda_i)$ be
multivalued maps between digital
images, $1 \le i \le v$.
Let $X=\Pi_{i=1}^v X_i$, $Y=\Pi_{i=1}^v Y_i$.
If the product multivalued map
\[ f=\Pi_{i=1}^v f_i: (X, T(\kappa_1,\ldots,\kappa_v)) \multimap (Y, T(\lambda_1, \ldots, \lambda_v)) \]
has strong continuity, then for each~$i$,
$f_i$ has strong continuity.
\end{thm}

\begin{proof}
Let $x_i \adj_{\kappa_i} x_i'$ in
$X_i$. Let
$p = (x_1,\ldots,x_v)$ and 
$p' = (x_1',\ldots,x_v')$.
Note $p \adj_{T(\kappa_1,\ldots,\kappa_v)} p'$ in $X$.
Since $f$ has strong continuity,
for every $q=(y_1,\ldots,y_v)\in f(p)=
\Pi_{i=1}^v f_i(x_i)$ where $y_i \in f_i(x_i)$, there
exists $q'=(y_1',\ldots,y_v')\in f(p')=
\Pi_{i=1}^v f_i(x_i')$ where $y_i' \in f_i(x_i')$
such that either $q=q'$ or
$q \adj_{T(\lambda_1,\ldots,\lambda_v)} q'$; and therefore
$y_i=y_i'$ for all $i$ or
$y_i \adj_{\lambda_i} y_i'$ for
all $i$. Also, for every $r'=(r_1',\ldots,r_v')\in f(p')$ where $r_i' \in f_i(x_i')$, there
exists $r=(r_1,\ldots,r_v)\in f(p)$ where $r_i \in f_i(x_i)$
such that either $r=r'$ or
$r \adj_{T(\lambda_1,\ldots,\lambda_v)} r'$; and therefore
$r_i=r_i'$ for all $i$ or
$r_i \adj_{\lambda_i} r_i'$ for
all $i$. Thus $f_i$ has 
$(\kappa_i,\lambda_i)$-strong continuity.
\end{proof}

The converse of Theorem~\ref{strong-tensor-cont-implies-factor} is not generally true, as shown by the following.

\begin{exl}
\label{tensor-non-strong-prod}
Let $f_1: ([0,1]_{\Z},c_1) \multimap ([0,1]_{\Z},c_1)$ be
defined by $f_1(x)=\{0\}$.
Let $f_2: ([0,1]_{\Z},c_1) \multimap ([0,1]_{\Z},c_1)$ be
defined by $f_2(x)=\{x\}$. Then
$f_1$ and $f_2$ both have strong
continuity. However, $f_1 \times f_2$ does not have
$(T(c_1,c_1),T(c_1,c_1))$-strong continuity.
\end{exl}

\begin{proof}
It is easily seen that $f_1$ and $f_2$ both have strong
continuity. However, in
Example~\ref{tensor-non-weak-prod}, we showed that
$f_1 \times f_2$ does not have
$(T(c_1,c_1),T(c_1,c_1))$-weak continuity. Therefore, $f_1 \times f_2$ does not have
$(T(c_1,c_1),T(c_1,c_1))$-strong continuity.
\end{proof}

\begin{thm}
\label{strong-Cart-prod-implies-factor}
Let $f_i: (X_i,\kappa_i) \multimap (Y_i,\lambda_i)$ be
multivalued maps between digital
images, $1 \le i \le v$.
Let $X=\Pi_{i=1}^v X_i$, $Y=\Pi_{i=1}^v Y_i$.
Then the product multivalued map
\[ f=\Pi_{i=1}^v f_i: (X, \times_{i=1}^v \kappa_i) \multimap (Y, \times_{i=1}^v \lambda_i) \]
has strong continuity if and
only if for each~$i$,
$f_i$ has strong continuity.
\end{thm}

\begin{proof}
Suppose $f$ has strong continuity.
Let $x_i \adj_{\kappa_i} x_i'$ in
$X_i$. Then
\[ p=(x_1,\ldots,x_v) \adj_{\times_{i=1}^v \kappa_i}
(x_1,\ldots,x_{j-1},x_j',x_{j+1},\ldots,x_v)=p'
\]
in $X$, for some index~$j$. Since $f$ has strong
continuity, we must have that for
every $q=(q_1,\ldots,q_v)\in f(p)$ there exists $q'=(q_1',\ldots,q_v') \in f(p')$ such that
$q = q'$ or $q \adj_{\times_{i=1}^v \lambda_i} q'$, so $q_i=q_i'$ or $q_i \adj_{\lambda_i} q_i'$; and for every
$r'=(r_1',\ldots,r_v') \in f(p')$ there exists
$r=(r_1,\ldots,r_v) \in f(p)$
such that $r=r'$ or
$r \adj_{\times_{i=1}^v \lambda_i} r'$, so $r_i=r_i'$ or $r_i \adj_{\lambda_i} r_i'$.
Therefore, $f_i$ has strong continuity.

Suppose for each~$i$, $f_i$ has
strong continuity. Let
$p=(x_1,\ldots, x_v)$ and 
$p'=(x_1',\ldots, x_v')$ with
$x_i,x_i' \in X_i$ be such that
$p \adj_{\times_{i=1}^v \kappa_i} p'$. Then for some
index $j$, $x_j \adj_{\kappa_j} x_j'$ and for all indices $i \neq j$, $x_i = x_i'$. Therefore,
$i \neq j$ implies there exists $q_i \in f_i(x_i)=f_i(x_i')$; and since
$f_j$ has strong continuity, for
every $q_j \in f_j(x_j)$ there
exists $q_j' \in f_j(x_j')$ such that $q_j=q_j'$ or $q_j \adj_{\lambda_j} q_j'$.
Let $q'=(q_1,\ldots,q_{j-1},q_j',q_{j+1},\ldots,q_v)$. Then
$q =(q_1, \ldots, q_v) =
q'$ or $q \adj_{\times_{i=1}^v \lambda_i} q'$ with $q \in f(p)$, $q' \in f(p')$. Similarly, for every $r' \in f(p')$ there exists
$r \in f(p)$ such that $r=r'$
or $r \adj_{\times_{i=1}^v \lambda_i} r'$.
Thus, $f$ has strong continuity.
\end{proof}

For the lexicographic adjacency,
the following shows there is not a
general product property for
weak or strong continuity.

\begin{exl}
\label{lex-weakStrong-exl}
Let $f_1: ([0,1]_{\Z},c_1) \multimap ([0,1]_{\Z},c_1)$ be
the multivalued function $f_1(x)=\{0\}$. Let $f_2: (\{0,2\},c_1) \multimap (\{0,2\},c_1)$ be the function
$f_2(x)=\{x\}$. Then $f_1$ and $f_2$ have weak continuity and strong continuity, but
$f_1 \times f_2$ lacks both $(L(c_1,c_1),L(c_1,c_1))$-weak continuity and $(L(c_1,c_1),L(c_1,c_1))$-strong continuity.
\end{exl}

\begin{proof} It is easy to see
that $f_1$ and $f_2$ have weak continuity and strong continuity, and that
$p=(0,0) \adj_{L(c_1,c_1)} (1,2)=p'$. However
\[ (f_1 \times f_2)(p)= \{(0,0)\}
\mbox{ and }
 (f_1 \times f_2)(p')= \{(0,2)\},
\]
are not $L(c_1,c_1)$-adjacent, so
$f_1 \times f_2$ lacks $(L(c_1,c_1),L(c_1,c_1))$-weak continuity and therefore lacks 
$(L(c_1,c_1),L(c_1,c_1))$-strong continuity.
\end{proof}

For the lexicographic adjacency,
the following shows there is not a
general factor property for
weak or strong continuity.

\begin{exl}
\label{lex-weakStrong-factor-exl}
Let $f_1: ([0,1]_{\Z}, c_1) \multimap ([0,1]_{\Z}, c_1)$
be the multivalued function
$f_1(x)= [0,1]_{\Z}$. Let
$f_2: ([0,1]_{\Z}, c_1) \multimap (\{0,2\}, c_1)$ be
the multivalued function
$f_2(x) = \{2x\}$. Then
$f_1 \times f_2: [0,1]_{\Z}^2 \multimap [0,1]_{\Z} \times \{0,2\}$ has $(L(c_1,c_1),L(c_1,c_1))$-weak and $(L(c_1,c_1),L(c_1,c_1))$-strong continuity, although
$f_2$ lacks both weak and strong continuity.
\end{exl}

\begin{proof}
It is easy to see that $f_2$ lacks weak and strong continuity. Since 
\[ (f_1 \times f_2)(0,0)=(f_1 \times f_2)(1,0)=\{(0,0), (1,0)\},
\]
\[ (f_1 \times f_2)(0,1)=(f_1 \times f_2)(1,1)=\{(0,2),(1,2)\},
\]
it follows easily that $f_1 \times f_2$ has both
$(L(c_1,c_1),L(c_1,c_1))$-weak continuity and $(L(c_1,c_1),L(c_1,c_1))$-strong continuity.
\end{proof}

\subsection{Continuous multifunctions}

\begin{lem}
\label{gcm-subdiv}
\rm{\cite{Boxer16a}}
Let $X \subset \Z^m$, $Y \subset \Z^n$.
Let $F: (X,c_a) \multimap (Y,c_b)$ be a
continuous multivalued function. Let
$f: (S(X,r),c_a) \to (Y,c_b)$ be a continuous function that
induces $F$. Let $s \in \N$. Then there is a
continuous function
$f_s: (S(X,rs),c_a) \to (Y,c_b)$ that induces $F$. $\qed$
\end{lem}

For the $NP_v$ adjacency, we have
the following.

\begin{thm}
\label{multi-prod-thm}
\rm{\cite{Boxer16a}}
Given multivalued functions
$F_i: (X_i,c_{a_i}) \multimap (Y_i,c_{b_i})$,
$1 \leq i \leq v$, each
$F_i$ is continuous if and only if the product multivalued function
\[ \Pi_{i=1}^v F_i: (\Pi_{i=1}^v X_i, NP_v(c_{a_1},\ldots, c_{a_v})) \multimap (\Pi_{i=1}^v Y_i, NP_v(c_{b_1}, \ldots, c_{b_v})) \]
is continuous. $\qed$
\end{thm}

For the tensor product, since a
single-valued function can be considered as multivalued,
Example~\ref{need-local-1-1} shows there is no general product rule for the continuity
of multivalued functions. 
However, we have the following.

\begin{thm}
\label{tensor-prod-multi-cont}
Let $F_i: (X_i,c_{a_i}) \multimap (Y_i,c_{b_i})$ be a continuous
multivalued function between digital images, $1 \le i \le v$.
Let $X=\Pi_{i=1}^v X_i$,
$Y=\Pi_{i=1}^v Y_i$,
$F=\Pi_{i=1}^v F_i: X \multimap Y$.
If for some positive integer~$r$ and for all~$i$ there is a continuous locally one-to-one
function $f_i: (S(X_i,r),c_{a_i}) \to (Y_i,c_{b_i})$ that generates
$F_i$, then $F$ is
$(T(c_{a_1},\ldots,c_{a_v}),T(c_{b_1},\ldots,c_{b_v}))$-continuous and is generated by a function
that is locally one-to-one.
\end{thm}

\begin{proof}
Let $f = \Pi_{i=1}^v f_i: \Pi_{i=1}^v S(X_i,r) \to Y$. It follows from
Theorem~\ref{T-prod-continuity} that
$f$ is $(T(c_{a_1},\ldots,c_{a_v}),T(c_{b_1},\ldots,c_{b_v}))$-continuous.
Further, given $q \in F(p)$
where $p=(x_1,\ldots, x_v)$ for
$x_i \in X_i$ and $q=(y_1,\ldots,y_v)$ where $y_i \in F_i(x_i)$, there exists
$x_i' \in S(\{x_i\},r) \subset S(X_i,r)$ such that $f_i(x_i')=y_i$.
Therefore, $f(x_1',\ldots,x_v')=q$.

For $w \adj_{T(c_{a_1},\ldots,c_{a_v})} w'$ in $S(X,r)=\Pi_{i=1}^v S(X_i,r)$, we
have $w=(w_1,\ldots,w_v)$ and
$w'=(w_1',\ldots,w_v')$,
where $w_i,w_i' \in S(X_i,r)$ and
$w_i \adj_{c_{a_i}} w_i'$. Since $f_i$ is locally
one-to-one and continuous, we have
$f_i(w_i) \adj_{c_{b_i}} f_i(w_i')$.
It follows that $f(w_1,\ldots,w_v) \adj_{T(c_{b_1},\ldots,c_{b_v})}
f(w_1', \ldots, w_v')$.
This allows us to conclude that $f$ is
$(T(c_{a_1},\ldots,c_{a_v}),T(c_{b_1},\ldots,c_{b_v}))$-continuous.
Thus, $f$ generates~$F$.

Let $p'=(x_1',\ldots,x_v') \adj_{T(\kappa_1,\ldots,\kappa_v)} p$ in $X$, where $x_i' \in X_i$.
Since $f_i$ is locally one-to-one,
$f_i(x_i) \adj_{\lambda_i} f_i(x_i')$ for all $i$. Therefore,
$f(p) \adj_{T(\lambda_1,\ldots,\lambda_v)} f(p')$, so $f$ is
locally one-to-one.
\end{proof}

Deciding whether the converse of
Theorem~\ref{tensor-prod-multi-cont} is true appears
to be a difficult problem.

For the Cartesian product adjacency, we have the following.

\begin{thm}
\label{tensor-factor-cont-implies-prod}
Let $F_i: (X_i,\kappa_i) \multimap (Y_i,\lambda_i)$ be a
multivalued function between digital images, where $\kappa_i=c_{a_i}$, $\lambda_i=c_{b_i}$, $1 \le i \le v$.
Let $X = \Pi_{i=1}^v X_i$,
$Y = \Pi_{i=1}^v Y_i$,
$F = \Pi_{i=1}^v F_i: X \multimap Y$. If each $F_i$ is continuous, then $F$ is $(\times_{i=1}^v \kappa_i,\times_{i=1}^v \lambda_i)$-continuous.
\end{thm}

\begin{proof}
Suppose each $F_i$ is continuous.
By Lemma~\ref{gcm-subdiv}, there
exists $r \in \N$ and generating
functions $f_i: S(X_i,r) \to Y_i$
of $F_i$.

We wish to show that $f=\Pi_{i=1}^v f_i$ generates $F$. Suppose $p \adj_{\times_{i=1}^v \kappa_i} p'$ in $S(X,r)$. Then
$p=(x_1,\ldots,x_v)$ and
$p'=(x_1',\ldots,x_v')$ where
$x_i,x_i' \in S(X_i,r)$ and
$x_i=x_i'$ for all but one index~$j$, with $x_j \adj_{\kappa_j} x_j'$. Since each
$f_i$ is $(\kappa_i,\lambda_i)$-continuous, we have
$f_j(x_j)=f_j(x_j')$ or
$f_j(x_j)\adj_{\lambda_j} f_j(x_j')$
and for all indices $i \neq j$ we have $f_i(x_i)=f_i(x_i')$.
Thus we have
$f(p)=f(p')$ or $f(p)\adj_{\times_{i=1}^v \lambda_i} f(p')$. Thus, $f$ is
$(\times_{i=1}^v \kappa_i, \times_{i=1}^v \lambda_i)$-continuous.

Let $y=(y_1,\ldots,y_v) \in F(X)$, where $y_i \in Y_i$. Then
there exists $x_i \in S(X_i,r)$
such that $f_i(x_i)=y_i$. For
$p=(x_1,\ldots,x_v)$, we have
$f(p)=(y_1,\ldots,y_v)$. Thus,
$f$ generates $F$, so $F$ is continuous.
\end{proof}

Deciding whether the converse of
Theorem~\ref{tensor-factor-cont-implies-prod} is true appears
to be a difficult problem.

For the lexicographic adjacency,
there is no general product rule for the
continuity of multivalued functions, as shown in Example~\ref{lexico-factor-not-implies-prod} (since a single-valued function can be regarded
as multivalued). However, we have
the following.

\begin{thm}
\label{multi-prod-cont-lex}
Let $F_i: (X_i,\kappa_i) \multimap (Y_i,\lambda_i)$ be a
continuous multivalued function
between digital images, $1 \le i \le v$. Let $X=\Pi_{i=1}^v X_i$,
$Y=\Pi_{i=1}^v Y_i$,
$F=\Pi_{i=1}^v F_i: X \multimap Y$. If each $F_i$ is generated by
a function $f_i: (S(X_i,r),\kappa_i)\to Y_i$ that
is locally one-to-one, then $F$
is $(L(\kappa_1,\ldots,\kappa_v),L(\lambda_1,\dots,\lambda_v))$-continuous.

\begin{proof}
By Theorem~\ref{lexico-1-1},
the single-valued function $f=\Pi_{i=1}^v f_i: \Pi_{i=1}^v S(X_i,r) \to Y$ is
$(L(\kappa_1,\ldots,\kappa_v),L(\lambda_1,\dots,\lambda_v))$-continuous. Further,
given $y=(y_1,\ldots,y_v) \in F(X)$ with $y_i \in Y_i$, there
exist $x_i' \in S(\{x_i\},r) \subset S(X_i,r)$ such
that $f_i(x_i')=y_i$. Therefore,
$y=f(x_1',\ldots,x_v')\in F(x_1,\ldots,x_v)$. Therefore,
$f$ generates $F$, and the
assertion follows.
\end{proof}
\end{thm}

The paper~\cite{egs08} has several
results concerning the following notions.

\begin{definition}
\rm{\cite{egs08}}
Let $(X,\kappa) \subset \Z^n$ be a digital image
and $Y \subset X$. We say that $Y$ is
a {\em $\kappa$-retract of $X$} if there exists a $\kappa$-continuous multivalued
function $F: X \multimap Y$ (a 
{\em multivalued $\kappa$-retraction}) such that
$F(y) = \{y\}$ if $y \in Y$.
\end{definition}

We generalize Theorem~\ref{ret-thm} as follows.

\begin{thm}
\label{multi-ret-thm}
\rm{\cite{Boxer16a}}
For $1 \leq i \leq v$, let $A_i \subset (X_i,\kappa_i) \subset \Z^{n_i}$.
Suppose $F_i: X_i \multimap A_i$ is a continuous multivalued function for all~$i$. Then $F_i$ is a multivalued retraction for all~$i$ if and only if
$F=\Pi_{i=1}^v F_i: \Pi_{i=1}^v X_i \multimap \Pi_{i=1}^v A_i$ is a multivalued
$NP_v(\kappa_1,\ldots,\kappa_v)$-retraction. $\qed$
\end{thm}

For the Cartesian product adjacency, we 
have the following.

\begin{thm}
\label{multi-retract-prod-Cart}
Let $r_i: X_i \multimap A_i$ be multivalued retractions, $1 \le i \le v$.
Let $X=\Pi_{i=1}^v X_i$, $A=\Pi_{i=1}^v A_i$, $r = \Pi_{i=1}^v r_i: X \multimap A$.
Then $r$ is a $\times_{i=1}^v \kappa_i$-multivalued retraction.
\end{thm}

\begin{proof}
Since $r_i$ is a multivalued retraction, we must have
that $r_i(X_i)=A_i$ and $r_i(a_i)=\{a_i\}$ for all $a_i \in A_i$. Therefore,
$r(X)=A$ and $r(a)=\{a\}$ for all $a \in A$. By Theorem~\ref{tensor-factor-cont-implies-prod}, $r$ is continuous, and 
therefore is a multivalued retraction. 
\end{proof}

\subsection{Connectivity preserving multifunctions}
\begin{thm}
\label{normal-conn-preserv}
\rm{\cite{Boxer16a}}
Let $f_i: (X_i, \kappa_i) \multimap (Y_i, \lambda_i)$ be
a multivalued function between digital images,
$1 \leq i \leq v$. Then the product map
\[ \Pi_{i=1}^v f_i : (\Pi_{i=1}^v X_i, NP_v(\kappa_1, \ldots, \kappa_v)) \multimap (\Pi_{i=1}^v Y_i, NP_v(\lambda_1, \ldots, \lambda_v))
\]
is a connectivity preserving multifunction if and
only if each $f_i$
is a connectivity preserving multifunction. $\qed$
\end{thm}

The tensor product adjacency does
not yield a similar result, as
shown in the following.

\begin{exl}
Consider $\{0\} \subset \Z$,
$[0,1]_{\Z} \subset \Z$. The
multivalued function
$f: (\{0\},c_1) \multimap ([0,1]_{\Z}, c_1)$
defined by $f(0)=[0,1]_{\Z}$
is connectivity preserving. However,
$f \times f: \{0\}^2 = \{(0,0)\}
\multimap [0,1]_{\Z}^2$ is not $(T(c_1,c_1),T(c_1,c_1))$-connectivity
preserving.
\end{exl}

\begin{proof}
This follows from the observations
that $\{(0,0)\}$ has a single point, hence must be $T(c_1,c_1)$-connected; but, by Example~\ref{conn-not-preserved},
$(f \times f)(0,0)=[0,1]_{\Z}^2$ is not $T(c_1,c_1)$-connected.
\end{proof}

However, we have the following.

\begin{thm}
Let $f_i: (X_i,\kappa_i) \multimap
(Y_i,\lambda_i)$ be multivalued
functions, $1 \le i \le v$. Let
$X = \Pi_{i=1}^v X_i$, $Y=\Pi_{i=1}^v Y_i$. Suppose
$f=\Pi_{i=1}^v f_i: X \multimap Y$
is $(T(\kappa_1,\ldots,\kappa_v),
T(\lambda_1, \ldots, \lambda_v))$-connectivity preserving. Then each $f_i$ is
connectivity preserving.
\end{thm}

\begin{proof}
Let $p=(x_1,\ldots,x_v) \in X$,
where $x_i \in X_i$. By
assumption, $f(p)=\Pi_{i=1}^v f_i(x_i)$ is $T(\lambda_1, \ldots, \lambda_v)$-connected. From Theorem~\ref{Pi-conn}, it follows
that $f_i(x_i)$ is $\lambda_i$-connected.

Suppose $x_i' \adj_{\kappa_i} x_i$ in $X_i$. Then $p'= (x_1',\ldots,x_v') \adj_{T(\kappa_1,\ldots,\kappa_v)} p$. Since $f$ is connectivity preserving, 
$f(p')$ and $f(p)$ are $T(\lambda_1,\ldots,\lambda_v)$-adjacent subsets of $Y$. This implies there exist
$q'=(y_1',\ldots,y_v') \in f(p')$, $q = (y_1, \ldots, y_v) \in f(p)$ such
that $q' \adj_{T(\kappa_1,\ldots,\kappa_v)} q$ or $q'=q$. Therefore,
for each index~$i$, $y_i' \adj_{\lambda_i} y_i$ or $y_i'=y_i$.
Since $y_i' \in f_i(x_i')$ and
$y_i \in f_i(x_i)$, we have that
$f_i(x_i')$ and $f_i(x_i)$ are
$\lambda_i$-adjacent subsets of
$Y_i$.

From Theorem~\ref{mildadj}, $f_i$
is connectivity preserving.
\end{proof}

For the Cartesian product adjacency,
we have the following.

\begin{thm}
Let $(X_i,\kappa_i)$ and $(Y_i,\lambda_i)$ be digital
images, for
$1 \le i \le v$. Let $f_i: X_i \multimap Y_i$ be a multivalued
function. Let $f = \Pi_{i=1}^v f_i: X=\Pi_{i=1}^v X_i \multimap Y=\Pi_{i=1}^v Y_i$ be the product 
function. Then $f$ is
$(\times_{i=1}^v \kappa_i,\times_{i=1}^v \lambda_i)$-connectivity preserving if and only if each $f_i$ is connectivity preserving.
\end{thm}

\begin{proof}
Suppose $f$ is connectivity preserving. Let $p = (x_1,\ldots,x_v) \in X$, where
$x_i \in X_i$. Then
$f(p)=\Pi_{i=1}^v f_i(x_i)$ is
$\times_{i=1}^v \lambda_i$-connected. By Theorem~\ref{product-projection-cont}, $f_i(x_i)=p_i(f(p))$ is
$\lambda_i$-connected.

For any given index $k$, let
$x_k \adj_{\kappa_k} x_k'$ in $X_k$. For all indices $i \neq k$,
let $x_i \in X_i$. Then 
$p=(x_1,\ldots,x_v)$ and
$p'=(x_1,\ldots,x_{k-1},x_k', x_{k+1}, \ldots, x_v)$ are
$\times_{i=1}^v \kappa_i$-adjacent.
Since $f$ is connectivity preserving,
$f(p)$ and $f(p')$ are
$\times_{i=1}^v \lambda_i$-adjacent
subsets of $Y$. Therefore,
Theorem~\ref{product-projection-cont} implies
$f_k(x_k)=p_k(f(p))$ and
$f_k(x_k')=p_k(f(p'))$ are
$\lambda_k$-adjacent
subsets of $Y_k$. It follows from
Theorem~\ref{mildadj} that $f_k$
is connectivity preserving. Since $k$ was an arbitrarily selected index,
$f_i$ is connectivity preserving for all~$i$.

Now suppose each $f_i$ is connectivity preserving. Let
$p=(x_1,\ldots,x_v) \in X$ where
$x_i \in X_i$. Then $f(p)=
\Pi_{i=1}^v f_i(x_i)$ is, by
Theorem~\ref{product-Pi-conn},
$\times_{i=1}^v \lambda_i$-connected.

Suppose $p \adj_{\times_{i=1}^v \lambda_i} p'$ in $X$. Then
for some index~$k$,
$x_k \adj_{\kappa_i} x_k'$
in $X_k$ and for $i \neq k$ there
exist $x_i \in X_i$ such that
\[ p=(x_1,\ldots,x_v),~~~
p'=(x_1,\ldots,x_{k-1}, x_k', x_{k+1}, \ldots,x_v).
\]
Since $f_k$ is connectivity preserving, there exist
$y_k \in f_k(x_k)$ and $y_k' \in f_k(x_k')$ such that 
$y_k \adj_{\lambda_k} y_k'$ or $y_k=y_k'$.
For $i \neq k$, let $y_i \in f_i(x_i)$. Then
$q=(y_1,\ldots,y_v) \in f(p)$ and
$q'=(y_1,\ldots,y_{k-1},y_k',y_{k+1},\ldots,y_v)  \in f(p')$ are $\times_{i=1}^v \lambda_i$-adjacent or equal. Therefore, $f(p)$ and $f(q)$ are
$\times_{i=1}^v \lambda_i$-adjacent
subsets of $Y$. It follows from
Theorem~\ref{mildadj} that $f$
is connectivity preserving.
\end{proof}

For lexicographic adjacency,
\begin{itemize}
\item Example~\ref{lexico-factor-not-implies-prod}
shows that there is no product property for
connectivity preservation; and
\item there is no factor property for connectivity preservation, as
      the following example shows.
\end{itemize}

\begin{exl}
Let $f_1:(\{0\},c_1) \multimap ([0,1]_{\Z},c_1)$ be the multivalued function
$f_1(0)=[0,1]_{\Z}$. Let
$f_2:(\{0\},c_1) \multimap (\{0,2\},c_1)$ be the multivalued function $f_2(0)=\{0,2\}$. Then
\[ f=f_1 \times f_2: \{0\}^2=\{(0,0)\} \multimap [0,1]_{\Z} \times \{0,2\} \]
is 
$(L(c_1,c_1),L(c_1,c_1))$-connectivity preserving, but $f_2$ is not
$(c_1,c_1)$-connectivity preserving.
\end{exl}

\begin{proof}
This follows from the observations that the single 
point $(0,0)$ is connected, and
$f(0,0)=[0,1]_{\Z} \times \{0,2\}$ is $L(c_1,c_1)$-connected (see Figure~2).
\end{proof}

\section{Shy maps}
We have the following.

\begin{thm}
Let $f: (X,\kappa) \to (Y,\lambda)$
be a shy map of digital images.
Then $f$ is an isomorphism if
and only if $f$ is locally one-to-one.
\end{thm}

\begin{proof}
It is obvious that if $f$ is an
isomorphism, then $f$ is locally one-to-one.

To show the converse, we argue as
follows. Since $f$ is shy, we know
$f$ is a continuous surjection.

To show $f$ is one-to-one, suppose
there exist $x,x' \in X$ such that $y=f(x)=f(x') \in Y$. Since $f$
is shy, $f^{-1}(y)$ is $\kappa$-connected. Therefore,
if $x \neq x'$ then there is a
path of distinct points
$P=\{x_i\}_{i=1}^m \subset f^{-1}(y)$
such that $x=x_1$, $x_i \adj x_{i+1}$ for $1 \le i < m$, and
$x_m=x'$. But since $f$ is locally
one-to-one, $f|_{N_{\kappa}^*(x)}$ is one-to-one, so $f(x_2) \neq f(x)$,
contrary to the assumption $P \subset f^{-1}(y)$. Therefore, we must have
$x=x'$, so $f$ is one-to-one.

Since $f$ is one-to-one, $f^{-1}$ is one-to-one. Since $f$ is shy,
given $y \adj y'$ in $Y$,
$f^{-1}(\{y,y'\})$ is connected.
Thus, $f^{-1}$ is continuous. This
completes the proof that $f$ is an
isomorphism.
\end{proof}

The following generalizes a result
of~\cite{Boxer16}.

\begin{thm}
\label{shy-prod}
\rm{\cite{Boxer16a}}
Let $f_i: (X_i, \kappa_i) \to (Y_i, \lambda_i)$ be
a continuous surjection between digital images,
$1 \leq i \leq v$. Then the product map
\[ \Pi_{i=1}^v f_i : (\Pi_{i=1}^v X_i, NP_v(\kappa_1, \ldots, \kappa_v)) \to (\Pi_{i=1}^v Y_i, NP_v(\lambda_1, \ldots, \lambda_v))
\]
is shy if and only if each $f_i$
is a shy map. $\qed$
\end{thm}

For the tensor product, we have
the following.

\begin{thm}
\label{T-prod-shy-implies-factor}
Let $f_i: (X_i, \kappa_i) \to (Y_i, \lambda_i)$ be
a surjection between digital images,
$1 \leq i \leq v$. Let
$X=\Pi_{i=1}^v X_i$, $Y=\Pi_{i=1}^v Y_i$.
If the product function
\[ f=\Pi_{i=1}^v f_i: (X, T(\kappa_1, \ldots, \kappa_v)) \to (Y, T(\lambda_1, \ldots, \lambda_v))
\]
is shy, then $f_i$ is shy for each~$i$.
\end{thm}

\begin{proof} 
Since $f$ is shy, it is continuous, so by Theorem~\ref{prod-cont-implies-factor}, each $f_i$ is continuous.
Clearly, each $f_i$ is a surjection.

Let $y_i \in Y_i$.
Let $y=(y_1,\ldots,y_v) \in Y$.
Since $f$ is shy, 
$f^{-1}(y)=\Pi_{i=1}^v f_i^{-1}(y_i)$ is $T(\kappa_1,\ldots,\kappa_v)$-connected. By Theorem~\ref{Pi-conn},
$f_i(y_i)$ is $\kappa_i$-connected.

Let $y_i' \adj_{\lambda_i} y_i$ in $Y_i$. Then $y'=(y_1',\ldots,y_v')
\adj_{T(\lambda_1,\ldots,\lambda_v)} y$. Since $f$ is shy,
\[f^{-1}(\{y,y'\})=f^{-1}(\{y\})\cup
  f^{-1}(\{y'\})=
  \Pi_{i=1}^v f_i^{-1}(y_i) \cup
  \Pi_{i=1}^v f_i^{-1}(y_i')
\]
is $T(\kappa_1,\ldots,\kappa_v)$-connected. By
Theorem~\ref{T-projection-cont},
\[p_i(f^{-1}(\{y,y'\}))=f_i^{-1}(y_i) \cup f_i^{-1}(y_i')\]
is $\kappa_i$-connected. From
Definition~\ref{shy-def}, we
conclude that $f_i$ is a shy map.
\end{proof}

The converse to Theorem~\ref{T-prod-shy-implies-factor} is not generally true, as shown by the following.

\begin{exl}
Let $f_1: ([0,1]_{\Z},c_1) \to (\{0\},c_1)$ be the function
$f_1(x)=0$. Let $f_2: ([0,1]_{\Z},c_1) \to ([0,1]_{\Z},c_1)$ be the function
$f_2(x)=x$. Then
$f_1$ and $f_2$ are shy, but
$f_1 \times f_2: ([0,1]_{\Z}^2, T(c_1,c_1)) \to (\{0\} \times [0,1]_{\Z}, T(c_1,c_1))$ is not shy.
\end{exl}

\begin{proof}
That $f_1$ and $f_2$ are shy is easily seen. Further, $f_1 \times f_2$ is a surjection. Notice that
$(0,0) \adj_{T(c_1,c_1)} (1,1)$, but
$(f_1 \times f_2)(0,0)=(0,0)$ and
$(f_1 \times f_2)(1,1)=(0,1)$ are
neither equal nor $T(c_1,c_1)$-adjacent. Therefore,
$f_1 \times f_2$ is not $(T(c_1,c_1),T(c_1,c_1))$-continuous, hence is not $(T(c_1,c_1),T(c_1,c_1))$-shy.
\end{proof}

For the Cartesian product adjacency, we have the following.

\begin{thm}
\label{Cart-shy-factor}
Let $f_i: (X_i, \kappa_i) \to (Y_i, \lambda_i)$ be
a surjection between digital images,
$1 \leq i \leq v$. Let
$X=\Pi_{i=1}^v X_i$, $Y=\Pi_{i=1}^v Y_i$.
Then the product function
\[ f=\Pi_{i=1}^v f_i: (X, \times_{i=1}^v \kappa_i) \to (Y, \times_{i=1}^v \lambda_i)
\]
is shy if and only if $f_i$ is shy for each~$i$.
\end{thm}

\begin{proof}
Suppose $f$ is shy. Then clearly
each $f_i$ is a surjection, and by
Theorem~\ref{Cart-prod-cont}, $f_i$
is continuous.

Let $y_i \in Y_i$. Let
$y=(y_1,\ldots, y_v) \in Y$.
Since $f$ is shy,
$f^{-1}(y)=\Pi_{i=1}^v f_i^{-1}(y_i)$ is $\times_{i=1}^v \kappa_i$-connected. By
Theorem~\ref{product-projection-cont}, the projection map $p_i$ is
continuous, so
$p_i(f^{-1}(y))=f_i^{-1}(y_i)$ is
$\kappa_i$-connected.

Let $y' \in Y$ be such that
$y' \adj_{\times_{i=1}^v \lambda_i} y$. Then $y'$ must be among the
points $q_i=(y_1,\ldots,y_{i-1},y_i',y_{i+1},\ldots,y_v)$, where
$y_i' \in Y_i$ satisfies $y_i' \adj_{\lambda_i} y_i$. Since $f$ is shy,
$f^{-1}(\{y,q_i\})=f^{-1}(y) \cup f^{-1}(q_i)$
is $\times_{i=1}^v \kappa_i$-connected. Since $p_i$ is
continuous,
\[ p_i(f^{-1}(\{y,q_i\}))=p_i(f^{-1}(y) \cup f^{-1}(q_i))=f_i^{-1}(y_i) \cup f_i^{-1}(y_i')=f_i^{-1}(\{y_i,y_i'\})
\]
is $\kappa_i$-connected. This completes the proof that each $f_i$ is shy.

Suppose each $f_i$ is shy. Then clearly $f$ is a surjection, and
by Theorem~\ref{Cart-prod-cont}, $f$ is continuous.

Let $y_i \in Y_i$. Let $y=(y_1,\ldots, y_v) \in Y$.
Since $f_i$ is shy, $f_i^{-1}(y_i)$ is $\kappa_i$-connected. By
Theorem~\ref{product-Pi-conn},
\begin{equation}
\label{shy-prod-1}
f^{-1}(y)=\Pi_{i=1}^v f_i^{-1}(y_i)
\mbox{ is } \times_{i=1}^v \kappa_i\mbox{-connected.}
\end{equation}

Let $y' \in Y$ be such that
$y' \adj_{\times_{i=1}^v \lambda_i} y$. Then for some index~$i$, 
$y'=(y_1,\ldots,y_{i-1},y_i',y_{i+1},\ldots,y_v)$, where
$y_i' \in Y_i$ satisfies $y_i' \adj_{\lambda_i} y_i$. Similarly,
\begin{equation}
\label{shy-prod-2}
f^{-1}(y') \mbox{ is }
\times_{i=1}^v \kappa_i\mbox{-connected.}
\end{equation}

Since $f_i$ is shy, $f_i^{-1}(\{y_i,y_i'\})$
is connected, so there exist $x_i \in f_i^{-1}(y_i)$,
$x_i' \in f_i^{-1}(y_i')$ such that
$x_i \adj_{\kappa_i} x_i'$ or $x_i = x_i'$. For indices $j \ne i$, 
let $x_j \in f_j^{-1}(y_j)$. Then
$w=(x_1,\ldots, x_v)$ and
$w'=(x_1,\ldots, x_{i-1},x_i',x_{i+1},\ldots, x_v)$ satisfy 
\begin{equation}
\label{shy-prod-3}
w \in f^{-1}(y),~ w' \in f^{-1}(y'), \mbox{ and }
w \adj_{\times_{i=1}^v \kappa_i} w'
\mbox{ or } w=w'.
\end{equation}
From statements~(\ref{shy-prod-1}),
(\ref{shy-prod-2}), and
(\ref{shy-prod-3}), we conclude that
$f^{-1}(\{y,y'\})$ is
$\times_{i=1}^v \kappa_i$-connected.
Therefore, $f$ is shy.
\end{proof}

For the lexicographic adjacency,
we have the following.

\begin{thm}
\label{lex-shy-prod}
Let $f_i: (X_i,\kappa_i) \to (Y_i,\lambda_i)$ be functions 
between digital images, $1 \le i \le v$. Let $X = \Pi_{i=1}^v X_i$,
$Y= \Pi_{i=1}^v Y_i$,
$f=\Pi_{i=1}^v f_i: (X,L(\kappa_1,\ldots,\kappa_v)) \to (Y,L(\lambda_1,\ldots,\lambda_v))$.
If each $f_i$ is shy, then $f$ is
shy.
\end{thm}

\begin{proof}
Let $y=(y_1,\ldots,y_v) \in Y$,
where $y_i \in Y_i$. Then
$f^{-1}(y)=\Pi_{i=1}^v f_i^{-1}(y_i)$. Since each $f_i$ is shy,
$f_i^{-1}(y_i)$ is $\kappa_i$-connected. By Theorem~\ref{lex-conn}, $f^{-1}(y)$ 
is $L(\kappa_1,\ldots,\kappa_v)$-connected.

Let $p=(y_1',\ldots,y_v') \adj_{L(\lambda_1,\ldots,\lambda_v)} y$ in $Y$. Then for
some smallest index~$k$, $y_k' \adj_{\lambda_k} y_k$ and if $k>1$ then $y_i=y_i'$ for $i<k$. Since
$f_k$ is shy, $f_k^{-1}(\{y_k,y_k'\})$ is $\kappa_k$-connected. Further, if
$k>1$ then $f_i^{-1}(\{y_i,y_i'\})=f_i^{-1}(y_i)$ is connected,
since $f_i$ is shy. Now,

\begin{equation}
\label{lex-shy-eqs0}
f^{-1}(p)= \Pi_{i<k}f_i^{-1}(y_i) \times f_k^{-1}(y_k) \times \Pi_{i>k}f_i^{-1}(y_i), 
\end{equation}
\begin{equation}
\label{lex-shy-eqs}
f^{-1}(p')= \Pi_{i<k}f_i^{-1}(y_i') \times f_k^{-1}(y_k') \times \Pi_{i>k}f_i^{-1}(y_i')
\end{equation}
By the shyness of the $f_i$ and Theorem~\ref{lex-conn}, each of
$f^{-1}(p)$ and $f^{-1}(p')$ is $L(\kappa_1,\ldots,\kappa_v)$-connected. Further,
since $y_i=y_i'$ for $i<k$ and, by shyness of $f_k$,
\begin{equation}
\label{is-conn}
f_k^{-1}(\{y_k,y_k'\}) \mbox{ is } \kappa_k\mbox{-connected,}
\end{equation}
from statements~(\ref{lex-shy-eqs0}), (\ref{lex-shy-eqs}), and~(\ref{is-conn})
we can conclude that
$f^{-1}(p)$ and $f^{-1}(p')$ are $L(\kappa_1,\ldots,\kappa_v)$-adjacent sets. Therefore,
$f^{-1}(\{p,p'\})=f^{-1}(p)\cup f^{-1}(p')$ is
$L(\kappa_1,\ldots,\kappa_v)$-connected.
Therefore, $f$ is shy.
\end{proof}

The following shows that the
converse of Theorem~\ref{lex-shy-prod} is not generally true.

\begin{exl}
Let $f_1: ([0,1]_{\Z},c_1) \to \{0\} \subset (\Z,c_1)$ be the function
$f_1(x)=0$. Let $f_2: (\{0,2\},c_1) \to \{0\} \subset (\Z,c_1)$ be the function
$f_2(x)=0$. Then
\[ f_1 \times f_2: ([0,1]_{\Z} \times \{0,2\}, L(c_1,c_1)) \to (\{(0,0)\},L(c_1,c_1)) \]
is shy, but $f_2$ is not shy.
\end{exl}

\begin{proof}
Since $f_2^{-1}(0)$ is not connected, $f_2$ is not shy. However,
$[0,1]_{\Z} \times \{0,2\}$ is $L(c_1,c_1)$-connected, as discussed in 
Example~\ref{pretzel}, so, from Definition~\ref{shy-def}, $f_1 \times f_2$ is shy.
\end{proof}

\section{Further remarks}
We have studied the tensor product, Cartesian product, and lexicographic adjacencies for finite Cartesian
products of digital images. We have obtained many results for
``product" and ``factor" properties that parallel
results obtained for extensions
of the normal product adjacency in~\cite{Boxer16a}.

However, there are many properties known~\cite{Boxer16a} for the
normal product adjacency whose analogs for the adjacencies studied
here are either false or we were not able to derive. By comparing the
results of~\cite{Boxer16a} with those of the current paper, it
appears that the normal product adjacency is the adjacency that
yields the most satisfying results for Cartesian products of digital images.

\section{Acknowledgment}
The anonymous reviewers were very helpful. Their
corrections and suggestions are gratefully acknowledged.


\begin{thebibliography}{11}

\bibitem{Berge}
C. Berge,
{\em Graphs and Hypergraphs}, 2nd edition, North-Holland, Amsterdam, 1976.

\bibitem{Borsuk}
K. Borsuk,
{\em Theory of Retracts},
Polish Scientific Publishers, Warsaw, 1967.

\bibitem{Boxer94}
L. Boxer,
Digitally Continuous Functions,
{\em Pattern Recognition Letters} 15 (1994), 833-839.

\bibitem{Boxer99}
L. Boxer,
A Classical Construction for the Digital Fundamental Group,
{\em Pattern Recognition Letters} 10 (1999), 51-62.

\bibitem{Boxer05}
L. Boxer,
Properties of Digital Homotopy,
{\em Journal of Mathematical Imaging and Vision} 22 (2005),
19-26.

\bibitem{Boxer06}
L. Boxer, Digital Products, Wedges, and Covering Spaces,
{\em Journal of Mathematical Imaging and Vision} 25 (2006), 159-171.

\bibitem{Boxer14}
L. Boxer,
Remarks on Digitally Continuous Multivalued Functions,
{\em Journal of Advances in Mathematics}
9 (1) (2014), 1755-1762.

\bibitem{Boxer16}
L. Boxer,
Digital Shy Maps,
{\em Applied General Topology}, to appear. Available at https://arxiv.org/abs/1606.00782

\bibitem{Boxer16a}
L. Boxer, Generalized Normal Product Adjacency in Digital Topology, submitted. Available at
http://arxiv.org/abs/1608.03204

\bibitem{BEKLL}
L. Boxer, O. Ege, I. Karaca, J. Lopez, and J. Louwsma, Digital Fixed Points, Approximate Fixed Points, and Universal Functions, 
{\em Applied General Topology}
17(2), 2016, 159-172.

\bibitem{BoxKar12}
L. Boxer and I. Karaca,
Fundamental Groups for Digital Products,
{\em Advances and Applications in Mathematical Sciences} 11(4) (2012), 161-180.

\bibitem{BoxSta16}
L. Boxer and P.C. Staecker,
Connectivity Preserving Multivalued Functions in Digital Topology,
{\em Journal of Mathematical Imaging and Vision} 55 (3) (2016), 370-377.
DOI 10.1007/s10851-015-0625-5

\bibitem{BoSt0}
L. Boxer and P.C. Staecker, Remarks on Pointed Digital Homotopy, submitted.  Available at
http://arxiv.org/abs/1503.03016


\bibitem{Chen94}
L. Chen, Gradually varied surfaces and its optimal uniform approximation, {\em SPIE Proceedings}
2182 (1994), 300-307.

\bibitem{Chen04}
L. Chen, {\em Discrete Surfaces and Manifolds}, Scientific Practical Computing, Rockville, MD, 2004

\bibitem {egs08}
C. Escribano, A. Giraldo, and M. Sastre,
``Digitally Continuous Multivalued Functions,''
in \emph{Discrete Geometry for Computer Imagery}, Lecture Notes in Computer Science, v. 4992, Springer,
2008, 81--92.

\bibitem{egs12} 
C. Escribano, A. Giraldo, and M. Sastre,
``Digitally Continuous Multivalued Functions, Morphological Operations and Thinning Algorithms,''
\emph{Journal of Mathematical Imaging and Vision} 42 (2012), 76--91.

\bibitem{gs15}
A. Giraldo and M. Sastre,
On the Composition of Digitally Continuous Multivalued Functions,
{\em Journal of Mathematical Imaging and Vision} 53 (2) (2015), 196-209.

\bibitem{Harary}
F. Harary,
On the composition of two graphs,
{\em Duke Mathematical Journal} 26 (1) (1959), 29-34.

\bibitem{Harary&Trauth}
F. Harary and C.A. Trauth, Jr.,
Connectedness of products of two directed graphs,
{\em SIAM Journal on Applied Mathematics} 14 (2) (1966), 250-254.

\bibitem{Han03}
S.-E. Han,
Computer topology and its applications,
{\em Honam Math. Journal} 25 (2003),
153-162.

\bibitem{Han05}
S.-E. Han,
Non-product property of the digital fundamental group,
{\em Information Sciences} 171 (2005), 73-91.

\bibitem{Khalimsky}
E. Khalimsky,
Motion, deformation, and homotopy in finite spaces, in
{\em Proceedings IEEE International Conference on Systems, Man, and Cybernetics},
1987, 227-234.

\bibitem{Kovalevsky}
V.A. Kovalevsky,
A new concept for digital geometry,
{\em Shape in Picture},
Springer-Verlag, New York, 1994, pp. 37-51.

\bibitem{Rosenfeld}
A. Rosenfeld,
`Continuous' Functions on Digital Images,
{\em Pattern Recognition Letters} 4 (1987), 177-184.

\bibitem{Sabidussi60}
G. Sabidussi,
Graph multiplication,
{\em Mathematische Zeitschrift} 72 (1960),
446-457.

\bibitem{Tsaur}
Tsaur, R., and Smyth, M.: ``Continuous" multifunctions in discrete spaces with applications to fixed point theory. In: Bertrand, G., Imiya, A., Klette, R. (eds.),
{\em Digital and Image Geometry}, 
Lecture Notes in Computer Science, vol. 2243, pp. 151-162. Springer Berlin / Heidelberg (2001),
http://dx.doi.org/10.1007/3-540-45576-0 5, 10.1007/3-540-45576-0 5

\bibitem{vL-W}
J.H. van Lint and R.M. Wilson,
{\em A Course in Combinatorics},
Cambridge University Press, Cambridge, 1992.
\end{thebibliography}
\end{document}